\documentclass[review]{elsarticle}

\usepackage{lineno,hyperref}
\modulolinenumbers[5]

\usepackage{amscd}
\usepackage{pst-node}
\usepackage{tikz-cd}
\usepackage{pgfplots}

\usepackage{mathtools}
\DeclarePairedDelimiter\ceil{\lceil}{\rceil}
\DeclarePairedDelimiter\floor{\lfloor}{\rfloor}

\journal{}

\usepackage{amsmath}
\usepackage{amsthm}
\usepackage{amssymb}

 \newtheorem{theorem}{Theorem}[section]
 \newtheorem{corollary}[theorem]{Corollary}
 \newtheorem{lem}[theorem]{Lemma}
 \newtheorem{proposition}[theorem]{Proposition}
 \newtheorem{definition}[theorem]{Definition}
 \newtheorem{remark}[theorem]{Remark}
 \newtheorem{example}[theorem]{Example}

\newcommand{\om}{\textnormal{\textbf{o}}}

\newcommand {\im}{ \mathop {\rm Im} \nolimits }

\begin{document}

\begin{frontmatter}

\title{A generalized metric-type structure with some applications}




\author[mymainaddress]{Hallowed O. Olaoluwa\corref{mycorrespondingauthor}}
\cortext[mycorrespondingauthor]{Corresponding author}
\ead{holaoluwa@unilag.edu.ng}

\author[mysecondaryaddress]{Aminat O. Ige}
\ead{aminat.ige@lasu.edu.ng}

\author[mymainaddress]{Johnson O. Olaleru}
\ead{jolaleru@unilag.edu.ng}

\address[mymainaddress]{Department of Mathematics, University of Lagos, Nigeria}
\address[mysecondaryaddress]{Department of Mathematics, Lagos State University, Nigeria}

\begin{abstract}
The paper introduces the class of O-metric spaces, a novel generalization of metric-type spaces, classifying almost all possible metric types into upward and downward O-metrics. We list some topologies arising from O-metrics and discuss convergence, sequential continuity, first countability and T$_2$ separation. The topology of an O-metric space can be generated by an upward O-metric on the space hence the focus on upward O-metric spaces. A theorem on the existence and uniqueness of a fixed point of some contractive-like map is proved and related with some other well known fixed point results in literature. Applications to the estimation of distances, polygon inequalities, and optimization of entries in infinite symmetric matrices  are also given.\end{abstract}


\end{frontmatter}


\section{Introduction}

The concept of metric spaces is fundamental in analysis, and many authors have worked around it to arrive at abridged concepts such as partial metric spaces, multiplicative metric spaces, b-metric spaces, etc. 
In several generalization attempts, the  axioms of a metric space have been perturbed to derive other spaces. In particular, the triangle inequality condition of a metric space is quite adjustable to accommodate other binary operations. One of such modifications (replacing the additive operation with a scaled addition) resulted in the concept of b-metric spaces (see \cite{Bakhtin1989,Czerwik1993}) defined below:

\begin{definition} \label {(b-Metric)}\textup{\cite{Czerwik1993}} 
Let $X$ be a non-empty set, and $s \geq 1$ a given positive real number. Suppose that the mapping $d: X\times X\longrightarrow [0,\infty)$ satisfies:\\
$(d_1)$ $d(x,y)=0$ if and only if $x=y;$\\
$(d_2)$ $d(x,y)=d(y,x)$ for all $x,y \in X;$\\
$(d_3)$ $d(x,y) \leq s[d(x,z)+d(z,y)]$ for all $x,y,z \in X.$\\
Then, $d$ is called a $b$-metric on $X$ and the triple $(X,d,s)$ is called a $b$-metric space.
\end {definition}
\noindent
Another modification, obtained by replacing addition with multiplication, resulted in a modification of the self distance axiom in the concept of
multiplicative metric spaces introduced in \cite{mult} as follows:
\begin{definition}\textup{\cite{mult}}
Let $X$ be a non-empty set. A function $d^*:X \times X\rightarrow [1,\infty)$ is a multiplicative metric on $X$ if it satisfies for all $x,y,z \in X$ the following:
\\
(i) $d^*(x,y) = 1$ $\Longleftrightarrow x=y$; \\
(ii) $d^*(x,y) = d^*(y,x)$; \\
(iii) $d^*(x,z)\leq d^*(x,y) d^*(y,z)$  (Multiplicative triangle inequality).\\
The pair $(X,d^*)$ is called a multiplicative metric space.
\end{definition}
\noindent
\noindent
In fact, authors in \cite{mgraph} introduced the concept of b-multiplicative metric spaces by combining the notions of b-metrics and multiplicative metrics.
However, while b-metric spaces proved to be a useful generalization of metric spaces, it is easy to note that multiplicative metric spaces are essentially metric spaces, and b-multiplicative metric spaces are b-metric spaces (see \cite{survey,Ige2021}). 
With the possibility of a generalization of the triangle inequality to accommodate many other operations, Khojasteh \textit{et al.} \cite{thetamet} introduced $\theta$-metric spaces as follows:

\begin{definition}\textup{\cite{thetamet}\label{action}
Let $ \theta: [0, \infty) \times [0, \infty) \to [0, \infty)$ be a continuous mapping with respect to each variable, such that: \\
(i) $\theta(0,0) = 0$ and $\theta(u,v) = \theta(v,u)$ for all ${u}, {v} \ge 0,$\\
(ii) $\theta(w,t) < \theta ({u},{v})$ if either $w<u$ and  $t \leq v$ or $w \leq u$ and  $t<v$,\\
(iii) for each $r \in \im (\theta)$ and for each $v \in [0,r],$ there exists ${u} \in [0,r]$ such that $\theta(u,v) = r,$ (here, $ \im (\theta) = \left \{\theta({u},{v}): {u} \ge 0, {v}\ge 0 \right \}$)\\
(iv) $\theta(u,0)\le u,$ for all $u>0.$ \\
%
Such a function $\theta$ is called a B-action. Let $X$ be a non-empty set. A mapping $d_{\theta}: X \times X \longrightarrow [0, \infty$) is called a $\theta-metric$ on ${X}$ with respect to the $\theta$ if $d_{\theta}$ satisfies the following: \\
$(A_1)$ $d_{\theta}(x,y)=0$ if and only if $x=y;$\\
$(A_2)$ $d_{\theta}(x,y)=d_{\theta}(y,x)$, for all $x,y \in X;$\\
$(A_3)$ $d_{\theta}(x,y)\leq \theta(d_{\theta}(x,z), d_{\theta}(z,y))$ for all $x,y,z \in X.$\\
A pair $(X, d_{\theta})$ is called a $\theta$-metric space. 
}
\end {definition}
\noindent
The concepts of b-metrics, multiplicative metrics and b-multiplicative metrics are however not captured by the properties of $\theta$-metrics and B-actions. Indeed, while a metric space is a $\theta$-metric space with $\theta(u,v)= u+v,$ for all $u,v \geq 0$, b-metrics with optimal scalar constant $s>1$ cannot be $\theta$-metric spaces because condition (iv) in Definition \ref{action}  fails, neither can multiplicative metrics be $\theta$-metrics because condition $(A_1)$ of Definition \ref{action} fails. This is a gap that needs to be filled which necessitates this study. Of recent, two notions were introduced as generalization of b-metric spaces: $p$-metric spaces (see \cite{pbest}) and controlled metric type spaces (see \cite{discover}).

\begin{definition}\textup{\cite{discover}
Let $X$ be a nonempty set and $\eta:X \times X \to [1,\infty)$. The function $d: X \times X \to [0,\infty)$ is a controlled metric type if
for all $x,y,z \in X$, the following conditions hold:\\
$(\Omega_1)$ $d(x,y)=0$ iff $x=y$,\\
$(\Omega_2)$ $d(x,y)=d(y,x)$,\\
$(\Omega_3)$ $d(x,y) \leq \eta(x,z)d(x,z)+ \eta(z,y)d(z,y)$.\\
In this case, the pair $(X,d)$ is called a controlled metric type space.}
\end{definition}

\begin{definition}\textup{\cite{pbest}\label{pmetric}
Let $X$ be a nonempty set. A function $\tilde{d}:X \times X \to [0,\infty)$ is a $p$-metric
if there exists a strictly increasing continuous function $\Omega:[0,\infty) \to [0,\infty)$
with $t \leq \Omega(t)$ for all $t \geq 0$ such that for all $x,y,z \in X$, the following conditions hold:\\
$(p_1)$ $\tilde{d}(x,y)=0$ iff $x=y$,\\
$(p_2)$ $\tilde{d}(x,y)=\tilde{d}(y,x)$,\\
$(p_3)$ $\tilde{d}(x,z) \leq \Omega(\tilde{d}(x,y)+\tilde{d}(y,z))$.\\
In this case, the pair $(X,\tilde{d})$ is called a $p$-metric space, or, an extended $b$-metric space.}
\end{definition}
\noindent
The class of $p$-metric spaces is larger than the class of $b$-metric spaces, since a $b$-metric is a $p$-metric, when $\Omega(t) = st$ for all $t \geq 0$.  Also, b-metric spaces are controlled metric type spaces, with 
$\eta(x,y)=s ~\forall x,y \in X$. However, these classes do not contain multiplicative metrics.  It will be very interesting to explore the possibility of unifying all the existing metric-type spaces, thus enriching the study of topological and analytical concepts including fixed point results (\cite{Olaleru2009}, \cite{Olaoluwa2015}, \cite{Olaoluwa2016}), and applications in inequalities, approximation, differential equations, etc.


\section{O-metric spaces: topological structure}

The abundance of metric-type spaces, especially metric spaces, multiplicative metric spaces and b-metric spaces leads to the generalization discussed below. In this sequel, $\mathbb{R}_+$ denotes the interval of real numbers $[0,\infty)$.

\begin{definition}\label{psimetric}
Let $X$ be a non-empty set and $\om:I_a \times I_a \to \mathbb{R}_+$ a function, where $I_a$ is an interval in $\mathbb{R}_+$ containing $a \in \mathbb{R}_+$. A function $d_\om: X \times X \to I_a$ is said to be an O-metric on $X$ and $(X,d_\om,a)$ an O-metric space,\footnote{For more precision, $X$ is also said to be an $\om$-metric space, and $d_\om$ an $\om$-metric so that the class of O-metric spaces represents the class of $\om$-metric spaces for some map $\om$.}
 if for all $x,y,z \in X$ the following hold:
\begin{itemize}
\item[(i)] $d_\om(x,y)=a$ if and only if $x=y$;
\item[(ii)] $d_\om(x,y)=d_\om(y,x)$;
\item[(iii)] $d_\om(x,z) \leq \om(d_\om(x,y),d_\om(y,z))$  (triangle $\om$-inequality). \footnote{When convenient, one can consider $\om$ as a binary operation on $\mathbb{R}_+$ and adopt the notation $d_\om(x,z) \leq d_\om(x,y) ~\om ~d_\om(y,z)$ for the triangle $\om$-inequality.}
\end{itemize}
\end{definition}
\noindent
We distinguish two main classes of O-metric spaces in the following definition:
\begin{definition}\label{classif} Let $(X,d_\om,a)$ an O-metric space, where $d_\om: X \times X \to I_a$ is an O-metric on the non-empty set $X$, with $\om:I_a \times I_a \to \mathbb{R}_+$ a function and $I_a$ an interval in $\mathbb{R}_+$ containing $a \in \mathbb{R}_+$. The O-metric space $(X,d_\om,a)$ is said to be:
\begin{itemize}
\item[(i)] an $a$-upward (or upward) O-metric space if $I_a \subset [a,\infty)$;
\item[(ii)] an $a$-downward (or downward) O-metric space if  $I_a \subset [0,a]$.
\end{itemize}
\end{definition}
\noindent
Without loss of generality, we take $I_a=[a,\infty)$ for $a$-upward O-metric spaces and $I_a=[0,a]$ for $a$-downward O-metric spaces. Indeed, if $I_a \subset [a,\infty)$, any $a$-upward (resp. $a$-downward) O-metric space $(X,d_\om,a)$ is also $a$-upward (resp. $a$-downward) with $d_{\tilde{\om}}=d_\om$, where $\tilde{\om}:[a,\infty) \times [a,\infty) \to [a,\infty)$ (resp. $\tilde{\om}:[0,a] \times [0,a] \to [0,a]$) is such that $\tilde{\om}=\om$ on $I_a \times I_a$.
\begin{remark}\label{classes} 
\textup{The class of (upward) O-metric spaces include known metric-type spaces like metric spaces, multiplicative metric spaces, and b-metric spaces:
\begin{enumerate}
\item If $a=0$, $I_0=[0,\infty)$ and $\om(u,v)=u+v$ for all $u,v \geq 0$, $(X,d_\om,a)$ is the metric space \footnote{A metric space is therefore a $0$-upward O-metric space.} and denoted by $(X,d_+,0)$ or simply $(X,d)$.
\item
If $a=0$, $I_0=[0,\infty)$ and $\om(u,v)=s(u+v)$ for all $u,v \geq 0$ and for some $s \geq 1$, $(X,d_\om,a)$ is the b-metric space. 
\item
If $a=1$, $I_1=[1,\infty)$ and $\om(u,v)=uv$ for all $u,v \geq 1$, $(X,d_\om,a)$ is the multiplicative metric space and is denoted by $(X,d_\times,1)$.
 \item
If $a=1$, $I_1=[1,\infty)$ and $\om(u,v)=(uv)^s$ for all $u,v \geq 1$ and for some $s \geq 1$, $(X,d_\om,a)$ is the b-multiplicative metric space.
\item
If $a=0$, $I_0=[0,\infty)$ and $\om(u,v)=\max\{u,v\}$ for all $u,v \geq 0$, $(X,d_\om,a)$ is an ultrametric space  (see \cite{ultra}).
\item
If $a=0$, $I_0=[0,\infty)$ and $\om$ is a B-action $\theta$, $(X,d_\om,a)$ is a $\theta$-metric space.
\item
If $a=0$, $I_0=[0,\infty)$ and $\om(u,v)=\Omega(u+v)$ for all $u,v \geq 0$, where $\Omega:[0,\infty) \to [0,\infty)$ is a strictly increasing continuous function with $t \leq \Omega(t)$ for all $t \geq 0$, $(X,d_\om,a)$ is a $p$-metric space.
\item
A controlled metric type  space $(X,d)$ is an O-metric space provided that $\forall t \geq 0$, $\eta(d^{-1}(t))$ is bounded, with $a=0$, $I_0=[0,\infty)$ and $\om(u,v)=\alpha(u)u+\alpha(v)v$ for all $u,v \geq 0$, where $\alpha:\mathbb{R}_+ \to [1,\infty)$ is defined by $\alpha(t)=\sup\{\eta(x,y): ~ d(x,y)=t\}$.
\end{enumerate}
}
\end{remark}
\noindent
The following are examples of O-metrics which are not known in literature:
\begin{example}\label{example2}
$X=\mathbb{R}$ with $d_\om(x,y)=\ln(1+|x-y|)$ for all $x,y \in X$, is an $a$-upward O-metric space where $a=0$, $I_0=[0,\infty)$ and $\om(u,v)=
(u+1)(v+1)$ for all $u,v \in I_0$. 
\end{example}
\begin{example}\label{genl}
In fact, for any function $\om:[a,\infty) \times [a,\infty) \to [a,\infty)$ increasing on each variables, and such that $\om(a,a)=a$, $u \leq \om(u,a)$  and $\om(u,v)=\om(v,u)$ for all $u,v \in [a,\infty)$, where $a \in \mathbb{R}_+$, the set $\mathbb{R}$ of real numbers is an O-metric space, with $d_{\om}(x,y)=\left\{\begin{array}{lllll} a & \mbox{ if } & x=y, \\ \om(f(x),f(y)) & \mbox{ if } & x\neq y, \end{array}\right.$ where $f(u)=a+|u-a|$ for $u \in \mathbb{R}$.
\end{example}
\begin{example}\label{example1}
$X=\mathbb{R}$ with $d_\om(x,y)=e^{-|x-y|}$ for all $x,y \in X$, is an $a$-downward O-metric space where $a=1$, $I_1=(0,1]$ and $\om(u,v)=\frac{u}{v}$ for all $u,v \in I_1$.
\end{example}
\noindent
The following example is that of an O-metric that is neither an upward or a downward O-metric.
\begin{example}\label{example3}
Let $X=\mathbb{R}$, with $d:X \times X \to (0,\infty)$ defined thus:
\begin{equation*}
d(x,y)=\left\{ \begin{array}{lll} e^{-|x-y|} & \mbox{ if } |x-y| \leq 1; \\ |x-y| & \mbox{ if } |x-y| > 1. \end{array} \right.
\end{equation*}
$(X,d,\om)$ is an O-metric space, with $a=1$, $I_1=(0,\infty)$ and $\om:I_1 \times I_1 \to I_1$ defined thus:
\begin{equation*}
\om(u,v)=\left\{ \begin{array}{lll} \max\left\{\dfrac{u}{v},-\ln(uv)\right\} & \mbox{ if } ~ 0<u,v \leq 1;
\\
\max\left\{ue^v,-\ln u+v\right\} & \mbox{ if } ~ 0<u\leq 1<v;
\\
\max\left\{\dfrac{e^{-u}}{v}, u-\ln v\right\} & \mbox{ if } ~ 0<v \leq 1<u;
\\
\max\left\{e^{-u+v}, u+v \right\} & \mbox{ if } ~ u,v >1.
\end{array}\right.
\end{equation*}
\end{example}
\noindent
\begin{remark}
If $d_{\om}$ and $d'_{\om}$ are two $\om$-metrics on a set $X$, with $\om:I_a \times I_a \to \mathbb{R}_+$ as in Definition \ref{psimetric}, then $\om(d_\om,d'_\om)$ is also an O-metric on $X$ provided that $\om(a,a)=a$, $\om$ is non-decreasing on both variables, and $(u ~\om ~v) \om (w ~\om ~ z)=(u ~\om~ w) \om (v ~\om~ z)$ for all $u,v,w,z \in I_a$. The last condition occurs whenever $\om$ is associative and commutative.
\end{remark}
\noindent
\subsection{Relating types of O-metric spaces}
\subsubsection{Generating new O-metrics from existing ones}
One can generate an uncountable number of O-metrics from an existing one on a given set $X$. The following proposition holds:
\begin{proposition}\label{unexpected}
Let $X$ be a non-empty set. Consider $a,b \in \mathbb{R}_+$ contained respectively in intervals $I_a$ and $J_b$ of non-negative real numbers. Let  $\theta:\mathbb{R}_+ \to \mathbb{R}_+$  be a non-decreasing function, bijective on $I_a$, with $\theta(a)=b$ and $\theta(I_a)=J_b$, and let $\om_a:I_a \times I_a \to \mathbb{R}_+$ and $\om_b:J_b \times J_b \to \mathbb{R}_+$ be two functions such that the following diagram commutes: 
%
%
\begin{equation}\label{diag}
\begin{CD}
J_b \times J_b @>\om_b>> \mathbb{R}_+\\
@AA\theta \times \theta A   @AA \theta A\\
I_a \times I_a @>\om_a >> \mathbb{R}_+
\end{CD}
\end{equation}
i.e. such that $$\om_b(\theta(u),\theta(v))=\theta(\om_a(u,v)) ~~ \forall u,v \in I_a.$$
Then given functions $d_{\om_a}: X \times X \to I_a$ and $d_{\om_b}: X \times X \to I_b$, with $d_{\om_b}=\theta \circ d_{\om_a}$,
\begin{itemize}
\item[(i)] $(X,d_{\om_b},b)$ is an O-metric space if $(X,d_{\om_a},a)$ is an O-metric space.
\item[(ii)] $d_{\om_b}$ is $b$-upward (resp. $b$-downward) if $d_{\om_a}$ is $a$-upward (resp. $a$-downward).
\item[(iii)] The converses of $(i)$ and $(ii)$ hold if $\theta$ is surjective.
\end{itemize}
\noindent
\end{proposition}
\begin{proof} 
\item (i) Suppose $(X,d_{\om_a},a)$ is an O-metric space. One can easily check that  $d_{\om_b}$ satisfies the first two axioms of an O-metric. For all $x,y,z \in X$, since diagram (\ref{diag}) is commutative,
%
$$
\begin{array}{lcl}
d_{\om_b}(x,z) = \theta(d_{\om_a}(x,z)) &\leq& \theta(\om_a(d_{\om_a}(x,y),d_{\om_a}(y,z)))
\\
&=& \om_b\left(\theta(d_{\om_a}(x,y)),\theta(d_{\om_a}(y,z))\right)
\\
&=& \om_b(d_{\om_b}(x,y),d_{\om_b}(y,z)).
\end{array}
$$
\noindent
Therefore $d_{\om_b}$ is an O-metric as well.
\item
(ii)
Suppose that $d_{\om_a}$ is an $a$-upward O-metric, i.e. $I_a \subset [a,\infty)$. For all $t \in J_b$, there is $r \in I_a$ such that $t=\theta(r)$. Now, $r \geq a \implies t =\theta(r) \geq \theta(a)=b$ hence $J_b \subset [b,\infty)$. Thus $d_{\om_b}$ is a $b$-upward O-metric. 
Similarly, if $d_{\om_a}$ is $a$-downward O-metric, then $d_{\om_b}$ is a $b$-downward O-metric. 
\item
(iii)
Suppose now that $\theta$ is surjective. Then one can construct a non-decreasing function $\theta^{-1}:\mathbb{R}_+ \to \mathbb{R}_+$, bijective on $J_b$ such that $\theta^{-1}(b)=a$, $\theta^{-1}(J_b)=I_a$, and the diagram obtained by replacing $\theta$ with $\theta^{-1}$ in (\ref{diag}) commutes. Since $d_{\om_a}=\theta^{-1} \circ d_{\om_b}$, it follows from (i) that $(X,d_{\om_a},a)$ is an O-metric space if $(X,d_{\om_b},b)$ is an O-metric space, and from (ii) that $d_{\om_a}$ is $a$-upward (resp. $a$-downward) if $d_{\om_b}$ is $b$-upward (resp. $b$-downward). 
\end{proof}
\noindent
A direct application of Proposition \ref{unexpected} is that positive powers of O-metrics are O-metrics: 
\begin{corollary}
Let $(X,d_\om,a)$ be an O-metric space and $r>0$. Then $(X,d_\phi,a^r)$ is an O-metric space, with $d_\phi=(d_\om)^r$ and $\phi(u,v)=\left(\om\left(u^{\frac{1}{r}},v^{\frac{1}{r}}\right)\right)^{r}$ for all $u,v \in I_a$.
\end{corollary}
\begin{proof}
In Proposition \ref{unexpected}, take $\theta(t)=t^r$ for all $t \in \mathbb{R}_+$.
\end{proof}

\subsubsection{Upward O-metric spaces and metric spaces}
We state more precisely how to generate upward O-metrics from metrics and vice versa. %
The proofs are obvious hence omitted. 
\begin{proposition}\label{metric to type}
Let $(X,d)$ be a metric space, and $a \in \mathbb{R}_+$. Consider two functions $\theta:[0,\infty) \to [a,\infty)$ and $\om:[a,\infty) \times [a,\infty) \to [a,\infty)$ such that:
\begin{itemize}
\item[$(\theta_1)$] $\theta(t)=a \Longleftrightarrow t=0$;
\item[$(\theta_2)$] $\theta$ is monotone non-decreasing;
\item[$(\theta_3)$] $\theta(t_1+t_2) \leq \om(\theta(t_1),\theta(t_2))$ for all $t_1,t_2 \geq 0$.
\end{itemize}
Then $(X,d_\om,a)$, where $d_\om=\theta \circ d$, is an $a$-upward O-metric space.
\end{proposition}
\noindent
The following proposition establishes conditions under which an upward O-metric generates a metric in the classical sense:
\begin{proposition}\label{type to metric}
Let $(X,d_\om,a)$ be an O-metric space, where $\om:[a,\infty) \times [a,\infty) \to [a,\infty)$, with $a \in \mathbb{R}_+$.  $X$ is a metric space with metric ${d}:X \times X \to [0,\infty)$ if there exists a monotone non-decreasing function $\lambda:[a,\infty)  \to [0,\infty)$ such that:
\begin{itemize}
\item[$(\lambda_1)$] $\lambda(t)=0 \Longleftrightarrow t=a$;
\item[$(\lambda_2)$] $\lambda(\om(u,v)) \leq \lambda(u)+\lambda(v)$ for all $u,v \geq a$;
\item[$(\lambda_3)$] ${d}=\lambda \circ d_\om$.
\end{itemize}
\end{proposition}

\noindent
Thus, the following equivalence theorem is obtained:

\begin{theorem}\label{metricequiv}
Let $X$ be a non-empty set. Let $d_\om: X \times X \to [a,\infty)$, $\lambda:[a,\infty) \to [0,\infty)$ and $\om:[a,\infty) \times [a,\infty) \to  [a,\infty)$, with $a \in \mathbb{R}_+$, such that:

\begin{itemize}
\item[$(E_1)$] $\lambda$ is increasing, with $\lambda(a)=0$;

\item[$(E_2)$] $\lambda(\om(u,v))=\lambda(u)+\lambda(v) ~ \forall u,v \geq a.$
\end{itemize}
Then $(X,d_\om,a)$ is an upward O-metric space if and only if $(X,\lambda \circ d_\om)$ is a metric space (or equivalently, $(X,d)$ is a metric space if and only if $(X,d_\om,a)$ is an upward O-metric space, with $d_\om=\lambda^{-1} \circ d$).
\end{theorem}

\begin{proof} From $(E_1)$, the function $\lambda:[a,\infty) \to Im(\lambda)$ is bijective hence the following holds: $\lambda^{-1}(\lambda(u))=u$ for all $u \geq a$. Suppose that $d_\om$ is an O-metric on $X$. Then from Proposition \ref{type to metric}, $\lambda \circ d_\om$ is a metric on $X$ since conditions $(\lambda_1)-(\lambda_3)$ are satisfied. (particularly, $(\lambda_1)$ is satisfied since $\lambda$ is strictly increasing and $\lambda (a)=0$). Conversely, suppose that $\lambda \circ d_\om$ is a metric on $X$. The function $\theta=\lambda^{-1}$ satisfies conditions $(\theta_1)-(\theta_3)$ of Proposition \ref{metric to type} hence $\theta \circ \lambda \circ d_\om =d_\om$ is an upward O-metric.
\end{proof}
\noindent

\noindent
From Theorem \ref{metricequiv}, one obtains the following corollary which establishes equivalence between metric spaces and multiplicative metric spaces.

\begin{corollary}\textup{\cite{survey}\label{multiplicativem}
$(X,d_\times,1)$ is a multiplicative metric space if and only if $(X,\ln\circ d_\times)$ is a metric space.}
\end{corollary}

\begin{proof}
Taking $a=1$, $\om(u,v)=uv$ for all $u,v \geq 1$, and $\lambda:[1,\infty) \to [0,\infty)$ defined by $\lambda(t)=\ln(t)$ for all $t \geq 1$, $\lambda$ is increasing, $\lambda(1)=0$ and $\lambda(\om(u,v))=\ln(uv)=\ln(u)+\ln(v)=\lambda(u)+\lambda(v)$ for all $u,v \in [1,\infty)$. The conditions $(E_1)$ and $(E_2)$ of Theorem \ref{metricequiv} are satisfied.
\end{proof}

\begin{example}\label{refev}
Given any metric space $(X,d)$ and a bijection $\varphi: \mathbb{R}_+ \to \mathbb{R}_+$ such that $\varphi(0)=0$, the composition $\varphi \circ d$ is an $\om$-metric on $X$, with $\om(u,v)=\varphi(\varphi^{-1}(u)+\varphi^{-1}(v))$, satisfying $(E_1)$ and $(E_2)$, with $\lambda=\varphi^{-1}$. 
\end{example}
\noindent

\subsubsection{Upward O-metric spaces and downward O-metric spaces}

The following proposition states amongst other things,  the relationship between some type of upward O-metrics and downward O-metrics. 
\begin{proposition}\label{summary}
Suppose $(X,d_\om,a)$ is an O-metric space, with $d_\om:X \times X \to I_a$ and $\om: I_a \times I_a \to \mathbb{R}_+$, where $I_a$ be an interval in $[0,\infty)$ containing $a \in \mathbb{R}_+$. Suppose there is a function $\phi: I_a  \times  I_a   \to \mathbb{R}$ such that
\begin{equation}\label{mirror}
w \leq \om(u,v) \Longleftrightarrow \phi(w,v) \leq u, ~\forall u,v,w \in I_a.
\end{equation}
Then,

\begin{itemize}
\item[(a)] The reverse triangle $\bar{\phi}$-inequality $\bar{\phi}(d_\om(x,z),d_\om(z,y)) \leq d_\om(x,y)$ holds for all $x,y,z \in X$, where $\bar{\phi}(w,v)=\max\{\phi(w,v),\phi(v,w)\}$.
\item[(b)] If $\theta:I_a \cup \im(\phi) \to \mathbb{R}_+$ is a continuous function decreasing on $I_a$, then $(X,d_\xi,\theta(a))$ is an O-metric space, with $d_\xi=\theta \circ d_\om$ and for all $u,v \in \im(\theta)$, $\xi(u,v)=\theta(\phi(\theta^{-1}(u),\theta^{-1}(v)))$.
\item[(c)] If, in addition to (b), $\theta([a,\infty)) \subset [0,a]$ with $\theta(a)=a$, then $(X,d_\om,a)$ is an upward O-metric space if and only if $(X,d_\xi,a)$ is a downward O-metric space.
\end{itemize}
\end{proposition}

\begin{proof}
\item
(a)  holds from equation (\ref{mirror}), the symmetry of $d_{\om}$, and the triangle $\om$-inequality.
\item
(b) Suppose $\theta: I_a \cup \im(\phi) \to \mathbb{R}_+$ is a continuous mapping decreasing on $I_a$. Since X is an O-metric space, 
$d_\xi(x,y)=\theta(a) \Longleftrightarrow \theta(d_\om(x,y))=\theta(a) \Longleftrightarrow d_\om(x,y)=a \Longleftrightarrow x=y$, and  $d_\xi(x,y)=\theta(d_\om(x,y))=\theta(d_\om(y,x))=d_\xi(y,x)$, for all $x,y \in X$.
\\ 
Taking the images by $\theta$ of the two sides of the inequality in (a), for all $x,y,z \in X$,
\begin{equation*}
\begin{array}{lcl}
\theta(d_\om(x,y)) &\leq&  \theta\left( \phi(d_\om(x,z),d_\om(z,y)) \right)
\\
d_\xi(x,y) & \leq &\xi(d_\xi(x,z),d_\xi(z,y)).
\end{array}
\end{equation*}
The continuity of $\theta$ ensures that $d_\xi$ has  its values in $\theta(I_a)$, an interval containing $\theta(a)$. Thus $(X,d_\xi,\theta(a))$ is an O-metric.
\item
(c) With the additional condition $\theta([a,\infty)) \subset [0,a]$ and $\theta(a)=a$, $(X,d_\om,a)$ is an upward O-metric space if and only if $d_\om(x,y) \geq a$ for all $x,y \in X$, if and only if $d_\xi(x,y)=\theta(d_\om(x,y)) \leq a$ for all $x,y \in X$, if and only if $(X,d_\xi,a)$ is a downward O-metric space.
\end{proof}
\noindent
\begin{example}\textbf{(Reverse triangle inequalities).} 
\item
(i) Any metric space $(X,d)$ satisfies (\ref{mirror}) with $a=0$, $I_a=[0,\infty)$, $\om(u,v)=u+v$ and $\phi(u,v)=u-v$. One then obtains $\bar{\phi}(u,v)=|u-v|$ and the reverse triangle $\bar{\phi}$-inequality is the usual reverse triangle inequality in a metric space:
\begin{equation}
|d(x,z)-d(z,y)| \leq d(x,y) ~\forall x,y,z \in X.
\end{equation}
(ii) Any b-metric space $(X,d,s)$ satisfies (\ref{mirror}) with $a=0$, $I_a=[0,\infty)$, $\om(u,v)=s(u+v)$, and $\phi(u,v)=\frac{u}{s}-v$. We then have $\bar{\phi}(u,v)=\left(\frac{1}{2}+\frac{1}{2s}\right)|u-v|-\left(\frac{1}{2}-\frac{1}{2s}\right)(u+v)$, with the reverse triangle $\bar{\phi}$-inequality: 
\begin{equation}
(s+1)|d(x,z)-d(z,y)|-(s-1)[d(x,z)+d(z,y)] \leq 2sd(x,y) ~~\forall x,y,z \in X.
\end{equation}
(iii) If $X$ is endowed with a multiplicative metric $d^*$, then (\ref{mirror}) is satisfied with $\om(u,v)=uv$ and $\phi(u,v)=\frac{u}{v}$, with $a=1$ and $I_a=[1,\infty)$. We have $\bar{\phi}(u,v)=\frac{u^2+v^2}{2uv}+\frac{|u^2-v^2|}{2uv}$ with the following reverse ``triangle multiplicative inequality" for all $x,y,z \in X$:
\begin{equation}
d^*(x,z)^2+d^*(z,y)^2+|d^*(x,z)^2-d^*(z,y)^2| \leq 2d^*(x,z)d^*(z,y)d^*(y,x).
\end{equation}
\end{example}

\noindent
Notice that condition (\ref{mirror}) can be replaced with some symmetric condition as follows:
\begin{proposition}
Suppose $(X,d_\om,a)$ is an O-metric space, with $d_\om:X \times X \to I_a$ and $\om: I_a \times I_a \to I_a$, where $I_a$ be an interval in $[0,\infty)$ containing $a \in \mathbb{R}_+$. Suppose there is a function $\phi: I_a  \times  I_a   \to \mathbb{R}$ such that $\om$ and $\phi$ are increasing in the first variable and $\phi(\om(u,v),v)=\om(\phi(u,v),v)=u$ for all $u,v \in I_a$. Then (a), (b), (c), of Proposition \ref{summary} hold.
\end{proposition}
\begin{proof}
Under the conditions of the proposition, the inequality (\ref{mirror}) in Proposition \ref{summary} holds, hence (a), (b), and (c) hold.
\end{proof}

\subsubsection{Finite product of O-metrics}
\noindent
The proposition below states that the Cartesian product of two O-metric spaces is an O-metric space. The proof is omitted. %

\begin{proposition}
Let $(X_1,d_{\om_1},a)$ and $(X_2,d_{\om_2},a)$ be two O-metric spaces, with $d_{\om_1}: X_1 \times X_1 \to I_a$ and $d_{\om_2}:X_2 \times X_2 \to I_a$, where $I_a$ is an interval in $\mathbb{R}_+$ containing $a \geq 0$, and $\om_1: I_a \times I_a \to \mathbb{R}_+$ and $\om_2:I_a \times I_a \to \mathbb{R}_+$ are the respective underlining functions. Let $X_1 \times X_2$ be the Cartesian product of $X_1$ and $X_2$. With $\om=\max\{\om_1,\om_2\}$, define the function $d_\om$ on $X_1 \times X_2$ by:
$$d_\om\left((x,y),(u,v) \right)=\phi(d_{\om_1}(x,u),d_{\om_2}(y,v)), ~\forall (x,y), (u,v) \in X_1 \times X_2,$$
where $\phi: I_a \times I_a \to \mathbb{R}_+$ is such that:
$$
\left\{
\begin{array}{lll}
\phi(u_1,v_1)=a \Leftrightarrow u_1=v_1=a,
\\
\phi \mbox{ is non-decreasing on both variables, }
\\
\phi(\om(u_1,u_2),\om(v_1,v_2)) \leq \om(\phi(u_1,v_1),\phi(u_2,v_2)), ~\forall u_1,u_2,v_1,v_2 \in I_a.
\end{array}
\right.
$$
Then $(X_1 \times X_2,d_\om,a)$ is an O-metric space. 
\end{proposition}
\noindent
It should be noted that if $\om_1$ and $\om_2$ are nondecreasing on both variables, and $d_{\om_1}$ and $d_{\om_2}$ are $a$-upward O-metrics, $\phi$ can be taken to be the maximum function. The proposition above can be generalized to a finite number of O-metric spaces as follows:
\begin{proposition}
Let $\{(X_i,d_{\om_i},a)\}_{1 \leq i \leq n}$ be a family of O-metric spaces, with O-metrics $d_{\om_i}: X_i \times X_i \to I_a$, where $I_a$ is an interval in $\mathbb{R}_+$ containing $a \geq 0$, and $\om_i: I_a \times I_a \to \mathbb{R}_+$ each of their respective underlining functions. The Cartesian product $X_1 \times X_2 \times \cdots \times X_n$ of the $X_i$'s is an O-metric space with the functions $\om=\max\{\om_1,\om_2, \ldots, \om_n\}$ and 
$$d_\om\left((x_1,x_2,\ldots,x_n),(y_1,y_2,\ldots,y_n) \right)=\phi(d_{\om_1}(x_1,y_1),d_{\om_2}(x_2,y_2),\ldots,d_{\om_n}(x_n,y_n)),$$
where $\phi: I_a \times I_a \times \cdots \times I_a \to \mathbb{R}_+$ is such that:
$$
\left\{
\begin{array}{lll}
\phi(u_1,u_2,\ldots,u_n)=a \Longleftrightarrow u_1=u_2=\ldots=u_n=a,
\\\phi \mbox{ is non-decreasing on all variables, }\\
\phi(\om(u_1,v_1),\om(u_2,v_2),\ldots,\om(u_n,v_n)) \leq \om(\phi(u_1,u_2,\ldots,u_n),\phi(v_1,v_2,\dots,v_n)),
\end{array}
\right.
$$
\noindent $\forall u_i,v_i  \in I_a ~ \forall i \in \{1,2,\ldots,n\}$.
\end{proposition}

\subsection{Topology induced by an O-metric}

\begin{definition}
Let $(X,d_\om,a)$ be an O-metric space. Given $x \in X$ and $r>0$, the open ball centered on $x$ and with radius $r$ is given by \[B(x,r)=\{y \in X: ~ |d_\om(x,y)-a|<r\}.\] 
\end{definition}
\noindent
If $X$ is an $a$-upward O-metric, i.e. if $I_a=[a,\infty)$,  then for $r>0$ and $x \in X$,  $B(x,r)=\{y \in X| ~ a \leq d_\om(x,y)<a+r\}$. This conforms to the definition of open balls in metric spaces and b-metric spaces (when $a=0$), and multiplicative metric spaces (when $a=1$). If $X$ is an $a$-downward O-metric, i.e. if $I_a=[0,a]$, $B(x,r)=\{y \in X| ~ a-r < d_\om(x,y) \leq a\}$ for $r>0$ and $x \in X$.
\\
\\
The following proposition lists and compares some topologies defined on an O-metric:

\begin{proposition}\label{topologiez}
An O-metric $d_\om:X \times X \to I_a$ generates the following topologies on $X$:
\begin{equation}
\mathcal{T}_1=\left\{A \subset X| ~ \forall x \in A ~\exists r \in (\im \om)\setminus \{a\},  ~B(x,|r-a|) \subset A \right\} \label{topologywithim}
\footnote{$\im \om=\left\{\om(u,v):~ u,v \in I_a\right\}$}
\end{equation}
\begin{equation}
\mathcal{T}_2=\left\{A \subset X: ~ \forall x \in A ~\exists r>0,~  B(x,r) \subset A \right\}
 \label{topology}
\end{equation}
\begin{equation}
\mathcal{T}_3=\left\{A \subset X: ~ \forall \{x_n\} \subset X ~ d_\om(x_n,x) \to a \mbox{ and } x \in A \implies \exists n_0 \in \mathbb{N}: x_n \in A ~~ \forall n \geq n_0\right\} \label{topologyseq}
\end{equation}

\noindent
We also have $\mathcal{T}_1 \subset \mathcal{T}_2 =\mathcal{T}_3.$
\end{proposition}
 \begin{proof}
\item
\textbf{(a)}
$\emptyset \in \mathcal{T}_1$ vacuously while $X \in \mathcal{T}_1$ trivially. Now, let $\{A_i: i \in I\}$ be a family of subsets of $\mathcal{T}_1$. For each $x \in \bigcup_{i \in I} A_i,$ there exists $j \in I$ such that $x \in A_j$ hence there is $r \in \im \om \setminus \{a\}$  satisfying $B(x,|r-a|) \subset A_j \subset \bigcup_{i \in I}A_i$. Thus $\bigcup_{i \in I}A_i \in \mathcal{T}_1.$
\\
Now, given two members $A_1$ and $A_2$ of $\mathcal{T}_1$, $x \in A_1 \cap A_2$ implies that there are radii $r_1,r_2 \in \im \om \setminus \{a\}$ such that $B(x,|r_i-a|)  \subset A_i$, $i \in \{1,2\}$, since $x \in A_1$ and $x \in A_2$. Taking $r_j \in \{r_1,r_2\}$ such that $|r_j-a|=\min\{|r_1-a|,|r_2-a|\}$, $B(x,|r_j-a|)  \subset A_1 \cap A_2$ hence $A_1 \cap A_2 \in \mathcal{T}_1$. We have thus proved that $\mathcal{T}_1$ is a topology on $X$. 
\\
\textbf{(b)} In fact, by substituting $\im \om$ by $\mathbb{R}_+$ and $|r-a|$ with $r$ in \textbf{(a)}, we obtain that $\mathcal{T}_2$ is also a topology on $X$.
\\
\textbf{(c)} 
Consider a family  $\{B_i: i \in I\}$ of subsets of $\mathcal{T}_3$, and a sequence $\{x_n\}_{n \in \mathbb{N}} \subset X$ such that $d_\om(x_n,x) \to a$, where $x \in \bigcup_{i \in I} B_i$. There exists $j \in I$ such that $x \in B_j \in \mathcal{T}_3$ hence there is $n_0 \in \mathbb{N}$ such that $\forall n \geq n_0$, $x_n \in B_j \subset \bigcup_{i \in I} B_i$. Thus $\bigcup_{i \in I} B_i \in \mathcal{T}_3$. In the case that $I=\{1,2\}$ and $\{y_n\}_{n \in \mathbb{N}}$ is a sequence in $X$ such that $d_\om(y_n,y) \to a$ with $y \in B_1 \cap B_2$, then $y \in B_1$ and $y \in B_2$ hence there are integers $n_1, n_2 \in \mathbb{N}$ such that $y_n \in B_1$ $\forall n \geq n_1$ and $y_n \in B_2$ $\forall n \geq n_2$; thus taking $n_0=\max\{n_1,n_2\}$, $y_n \in B_1 \cap B_2$ for all $n \geq n_0$. Thus $B_1 \cap B_2 \in \mathcal{T}_3$. Thus $\mathcal{T}_3$ is a topology on $X$.
\\
The inclusion $\mathcal{T}_1 \subset \mathcal{T}_2$ is obvious.\\
Let $A \in \mathcal{T}_2$. Then for each $x \in A$, there is $r>0$ such that $B(x,r) \subset A$. Let $(x_n)$ be a sequence in $X$ such that $d_\om(x_n,x) \to a$ with $x \in A$. There is $n_0 \in \mathbb{N}$ such that $|d_\om(x_n,x)-a|< r$ for all $n \geq n_0$, hence $x_n \in B(x,r) \subset A$. Thus $A \in \mathcal{T}_3$. Conversely, let $A \in \mathcal{T}_3$. Suppose, for contradiction, that $A \notin \mathcal{T}_2$. Then there is $x \in A$ such that for all $n \in \mathbb{N}$, $B(x,\frac{1}{n}) \not\subset A$. For each $n \in \mathbb{N}$ choose $x_n \in B(x,\frac{1}{n})$ such that $x_n \notin A$. For all $n \in \mathbb{N}$, $|d_\om(x_n,x)-a|< \frac{1}{n}$ hence $d_\om(x_n,x) \to a$. Since $A \in \mathcal{T}_3$, there is $n_0 \in \mathbb{N}$ such that $x_n \in A$ for all $n \geq n_0$, a contradiction. Thus $A \in \mathcal{T}_2$ and therefore $\mathcal{T}_2=\mathcal{T}_3$.
\end{proof}
\noindent
In the remaining part of this paper, the O-metric space $(X,d_\om,a)$ is endowed with the topology $\mathcal{T}=\mathcal{T}_2=\mathcal{T}_3$ in (\ref{topology}) and (\ref{topologyseq}), and is called the topology \textit{induced} by the O-metric $d_\om$, or simply the \textit{O-metric topology}. 

\subsubsection{Convergence and sequential continuity}

We have the following definition:
\begin{definition}
Given an O-metric $d_\om: X \times X \to I_a$ where $I_a$ is an interval of $\mathbb{R}_+$ containing $a \geq 0$, with $\om: I_a \times I_a \to [0,\infty)$. A sequence $\{x_n\}_{n \in \mathbb{N}}$ of points in $X$:
\\
(i)  converges to $x \in X$ if and only if $\lim_{n \to \infty} d_\om(x_n,x)=a$;
\\
(ii) is a Cauchy sequence if $\lim_{n,m \to \infty} d_\om(x_n,x_m)=a$.
\\
The space $(X,d_\om,a)$ is said to be complete if every Cauchy sequence of $X$ converges in $X$. 
\end{definition}
\noindent
From the above definition, the following proposition on sequential continuity immediately follows:
\begin{proposition}\label{sequential}
Let $(X_1,d_{\om_1},a)$ and $(X_2,d_{\om_2},b)$ be two O-metric spaces, and $f:X_1 \to X_2$ a map. The following are equivalent:
\begin{itemize}
\item[(i)] $f$ is sequentially continuous at $\tilde{x} \in X_1$ i.e. for any sequence $\{x_n\}_{n \in \mathbb{N}}$ of points in $X_1$, $x_n  \to \tilde{x} \implies f(x_n) \to f(\tilde{x})$;
\item[(ii)] $\forall \epsilon>0$, $\exists \delta>0:$ $|d_{\om_1}(x,\tilde{x})-a| < \delta \implies |d_{\om_2}(f(x),f(\tilde{x}))-b|<\epsilon$.
\end{itemize}
\end{proposition}

\begin{example}
Consider two O-metric spaces $(X_1,d_{\om_1},a)$ and $(X_2,d_{\om_2},b)$ such that $X_1=X_2=\mathbb{R}$, $a=1$, $b=0$, $\om_1(u,v)=\dfrac{u}{v}$ for all $u \geq 0$ and $v>0$,  $\om_2(u,v)=(u+1)(v+1)$ for all $u,v \geq 0$, $d_{\om_1}(x,y)=e^{-|x-y|}$ and $d_{\om_2}(x,y)=\ln(1+|x-y|)$ for all $x,y \in \mathbb{R}$. The map $f:X_1 \to X_2$ defined by $f(x)=x^2-2$ is sequentially continuous at any $\tilde{x} \in \mathbb{R}$ in that it satisfies (ii) of Proposition \ref{sequential}, with  $\delta=1-e^{|\tilde{x}|-\sqrt{|\tilde{x}|^2+e^\epsilon-1}}$ for any $\epsilon>0$.
\end{example}
\noindent
Next, we specify conditions under which the limit of a sequence, if it exists, is unique.
\begin{proposition}\label{UUU}
Let $(X,d_\om,a)$ be an O-metric space, with $\om:I_a \times I_a \to [0,\infty)$ and $d_\om:X \times X \to I_a$, where $I_a$ is an interval in $\mathbb{R}_+$ containing $a$. If $\{x_n\}_{n \in \mathbb{N}}$ is a sequence of points in $X$ that converges to a point $x \in X$, then $x$ is unique under the conditions $(U_1)$ and $(U_2)$, or $(U_1)$ and $(U_2')$, where:
\begin{itemize}
\item[$(U_1)$]
$\om$ is continuous at $(a,a)$; 
\item[$(U_2)$]
$\om$ is nondecreasing in both variables and either $\om(u,a)=a \Leftrightarrow u=a$ or $\om(a,u)=a \Leftrightarrow u=a$.
\item[$(U_2')$]
$\om$ is nondecreasing in one variable, say $u_i$, and for $u_j=a$ with $j \neq i$, $\om(u_1,u_2)=a \Leftrightarrow u_i=a$.
\end{itemize}
\end{proposition}
\begin{proof} Assume $(U_1)$ and $(U_2')$ hold (since $(U_2)$ implies $(U_2')$). From condition $(U_2')$, $\om(a,a)=a$. Now, suppose $\lim_{n \to \infty} d_\om(x_n,x)=\lim_{n \to \infty} d_\om(x_n,y)= a$, where $x,y \in X$.  $d_\om(x,y) \leq \om(d_\om(x,x_n),d_\om(x_n,y)) \to \om(a,a)=a$ since $\om$ is continuous; hence $d_\om(x,y) \leq a$. Also,
$d_\om(x_n,x) \leq \om(d_\om(x,y),d_\om(x_n,y))$ hence taking the limits, $a \leq \om(d_\om(x,y),a)$. Similarly, we have that $a \leq \om(a,d_\om(x,y))$. Thus, from condition $(U_2')$, $a \leq \om(u_1,u_2) \leq \om(a,a)=a$, with variable $u_i$ in which $\om$ is monotone nondecreasing such that $u_i=d_\om(x,y)$ and $u_j =a$ for $j \neq i$. $(U_2')$ then also implies that $d(x,y)=a$ so that $x=y$.
\end{proof}
\noindent
\noindent
From Remark \ref{classes}, one can easily verify that metric spaces, b-metric spaces, multiplicative metric spaces, b-multiplicative metric spaces, ultrametric spaces, $\theta$-metric spaces and p-metric spaces satisfy condtions $(U_1)$ and $(U_2)$, hence the limit of a convergent sequence in any of these spaces converge with respect to the O-metric topology. Also, the O-metric space in Example \ref{genl} satisfy conditions $(U_1)$ and $(U_2)$. However, not every O-metric space satisfies conditions $(U_1)$ and $(U_2)$ (or $(U_2')$) as shown in the example below:
\begin{example}\label{olala}
Consider the O-metric space $(X,d_\om,a)$ where $X=[-1,1]$, $~~a=1$, $d_\om(x,y)=\left\{\begin{array}{lll} 1 & \mbox{ if } x=y \\ |xy| & \mbox{ if } x \neq y \end{array}\right.$ and $\om(u,v)=\left\{\begin{array}{lll} \frac{1}{uv} &\mbox{ if } u,v \neq 0, \\ 1 &\mbox{ if } u=0 \mbox{ or } v=0. \end{array}\right.$ Consider the sequence $\{x_n\}_{n \in \mathbb{N}}$ defined by $x_n=1-\frac{1}{n}$, for all $n \in \mathbb{N}$. For $x \in \{-1,1\}$, $d_\om(x_n,x)=\left(1-\frac{1}{n}\right)|x|=1-\frac{1}{n} \to 1$ as $n \to \infty$. Hence $\{x_n\}$ has two limits, $\pm 1$. \end{example}
\noindent
Although the function $\om$ is not unique for any O-metric space  $(X,d_\om,a)$ (given that any other function which dominates $\om$ would qualify), the O-metric space in Example \ref{olala} is such that no other function $\om$ can satisfy conditions $(U_1)$ and $(U_2)$ (or $(U_2')$) since the limit of a convergent sequence needs not be unique. It should also be noted that conditions $(U_1)$ and $(U_2)$ (or $(U_2')$) are only sufficient conditions.

\subsubsection{Openness of balls and Hausdorff property}

It should be noted that authors in \cite{openball} proved that not all open balls are open sets in b-metric spaces, which, as explained earlier, are O-metric spaces. Therefore, further conditions are necessary for every open ball in an O-metric to be an open set.

\begin{theorem}\label{complicated}
Let $(X,d_\om,a)$ be an O-metric space, with $d_\om:X \times X \to I_a$ and $\om: I_a \times I_a \to \mathbb{R}_+$ where $I_a$ is an interval in $\mathbb{R}_+$ containing $a$.
Suppose that:
\begin{equation}\label{complicated1}
\left|
\begin{array}{lll}
\forall r>0 ~~ \forall u \in I_a  ~~|u-a|<r \implies \exists s>0 \mbox{ such that } |w-a|< r \mbox{ whenever } 
\\
w \in I_a \mbox{ with } w \leq \om(u,v)  \mbox{ for some } v \in I_a \mbox{ such that } |v-a|<s. 
\end{array}
\right.
\end{equation}
\\
Then every open ball  in $X$ is an open set.
\end{theorem}

\begin{proof}
Consider the open ball $B(x_0,r)$, where $x_0 \in X$ and $r>0$.
Let $x \in B(x_0,r)$. Then $|d_\om(x,x_0)-a|<r$.  Letting $u=d_\om(x,x_0)$, $u \in I_a$ and $|u-a|<r$  hence there is $s>0$ such that $|w-a|< r$ whenever $w \in I_a$ with $w \leq \om(u,v)$ for some $v \in I_a$ such that $|v-a|<s$. 
\\
Let $y \in B(x,s)$. $|d_\om(x,y)-a|<s$ and from the $\om$-triangle inequality, $d_\om(x_0,y) \leq \om(d_\om(x_0,x),d_\om(x,y))$. Thus, taking $v=d_\om(x,y)$ and $w=d_\om(x_0,y)$, we have that $v,w \in I_a$, $|v-a|<s$ and $w \leq \om(u,v)$, hence $|d_\om(x_0,y)-a| =|w-a|<r$ and $y \in B(x_0,r)$. This shows that there is $s>0$ such that $B(x,s) \subset B(x_0,r)$ and thus, $B(x_0,r) \in \mathcal{T}_3$.
\end{proof}
\noindent
The condition (\ref{complicated1}) in Theorem \ref{complicated} could be explained geometrically. Indeed, if we define an open $a$-centered slice of $I_a$ as $S_r:=I_a \cap (a-r,a+r)$ for some $r>0$, the condition simply means that for any open $a$-centered slice $S_r$ of $I_a$ and an element $u$ in it, there is another open $a$-centered slice $S_s$ of $I_a$ such that $S_r$ fully contains the portion of $I_a$ whose elements are not greater than some element in the image $\om(\{u\} \times S_s)$ of the strip $\{u\} \times S_s$ in the rectangle $S_r \times S_s$. 
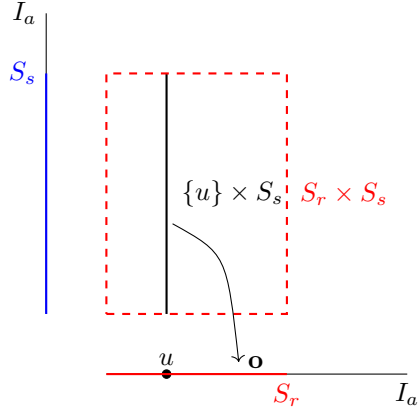
\begin{figure}[h]
\centering
\begin{tikzpicture}[scale=.8]
\draw (5,0) -- (10,0) node[anchor=north] {$I_a$}; 
\filldraw[black] (6,0) circle (2pt) node[anchor=south] {$u$}; 
\draw[thick] (6,1) --node[align=right] {$~~~~~~~~~~~~~~~\{u\} \times S_s$} (6,5) ; 
\draw[red,thick] (5,0) -- (8,0) node[anchor=north] {$S_r$}; 
\draw (4,1) -- (4,6) node[anchor=east] {$I_a$}; 
\draw[blue,thick] (4,1) -- (4,5) node[anchor=east] {$S_s$}; 
\draw[red,thick,dashed] (5,1) -- (8,1) -- node[] {$~~~~~~~~~~~~~S_r \times S_s$} (8,5) -- (5,5) -- (5,1); 
\draw[->] (6.1,2.5) .. controls (7,2) .. (7.2,0.2) node[anchor=west] {$\om$}; 
\end{tikzpicture}
\caption{Illustration of condition (\ref{complicated1}).}
\end{figure}

\noindent
In the case of upward O-metrics, we have the following: 
\begin{corollary}\label{c1c2}
Let $(X,d_\om,a)$ be an upward O-metric space, with $d_\om:X \times X \to I_a$ and $\om:I_a \times I_a \to I_a$, where $I_a=[a,\infty)$. 
If the following conditions hold:
\begin{itemize}
\item[$(C_1)$] There exists $\gamma:[a,\infty) \times [a,\infty) \to \mathbb{R}$ such that
$a \leq u <r \implies \left\{ \begin{array}{lll} \gamma(r,u)>a \\ \om(u,\gamma(r,u))\leq r,\end{array}\right.$
\item[$(C_2)$] $\om$ is increasing in both variables,
\end{itemize}
then every open ball is an open set.
\end{corollary}

\begin{proof} We note that given $r>0$ and $u \in [a,\infty)$ such that $|u-a|<r$, we have that $a \leq u <a+r$ hence taking $s=\gamma(a+r,u)-a$, from $(C_1)$, $s>0$. If $v \in [a,\infty)$ and $|v-a|<s$, then $a \leq v < a+\gamma(a+r,u)$; hence, given $(C_2)$, if $w \in [a,\infty)$ and $w \leq \om(u,v)$, $w \leq \om(u,\gamma(a+r,u)) \leq a+r$ so $|w-a|<r$. Thus condition (\ref{complicated1}) is satisfied.
\end{proof}
\begin{remark}
Metric spaces are upward O-metric spaces satisfying conditions $(C_1)$ and $(C_2)$ with $a=0$, $\om$ the addition, and $\gamma(r,u)=r-u$ for all $r,u \geq 0$. Conditions $(C_1)$ and $(C_2)$ hold in multiplicative metric spaces with $a=1$, $\om$ the product, and $\gamma(r,u)=\frac{r}{u}$ for all $r,u \geq 1$. Any $\theta$-metric space (see Definition \ref{action}) such that $\theta$ is a surjective B-action satisfies conditions $(C_1)$ and $(C_2)$, with $a=0$, $\om=\theta$, and $\gamma$ being the B-inverse action of $\theta$ (see \cite{thetamet}). 
\\
No b-metric space $(X,d,s)$ such that the optimal coefficient $s$ of the relaxed triangle inequality is strictly greater than $1$ satisfies condition $(C_1)$. Indeed, with $\om(u,v)=s(u+v)$,  if $u \in \left[\frac{r}{s},r\right)$, then $\om(u,\gamma(r,u)) \leq r \Rightarrow \gamma(r,u) \leq \frac{r}{s}-u \leq 0$.  
It is a conjecture whether in such case, there is always an open ball that is not an open set. It is easy to show that the conjecture is true for b-metrics $d(x,y)=|x-y|^p$ on $\mathbb{R}$, and $d(x,y)=\sum_{i=1}^{\infty} |x_i-y_i|^p$ on $l_p$ spaces, with $p>1$.
\end{remark}

\begin{example}
Consider as in Example \ref{example1}, the downward O-metric space $X=\mathbb{R}$, with $d_{\om}(x,y)=e^{-|x-y|}$, $a=1$, $I_1=(0,1]$, and $\om(u,v)=u/v$. It satisfies the general condition (\ref{complicated1}) of openness of open balls: 
$\forall r>0$ $\forall u \in (0,1]$ such that $|u-1|<r$, there is $s>0$ ($s=r$ if $r \geq 1$ and $s$ any positive number if $r<1$) such that $|w-1|<r$ whenever $w \in (0,1]$ with $w \leq u/v$ for some $v \in (0,1]$ such that $|v-1|<s$.
\end{example}
\noindent

\begin{proposition} Let $(X,d_\om,a)$ be an upward O-metric space, with $d_\om:X \times X \to I_a$ and $\om:I_a \times I_a \to I_a$, where $I_a=[a,\infty)$. Under conditions $(C_1)$ and $(C_2)$, the topology $\mathcal{T}$ is first countable and Hausdorff, and $\mathcal{T}=\mathcal{T}_1=\mathcal{T}_2=\mathcal{T}_3$.
\end{proposition}

\begin{proof} Under $(C_1)$ and $(C_2)$, open balls are open sets hence, for all $x \in X$, the collection $\left\{B\left(x,\dfrac{1}{n}\right): ~n \in \mathbb{N}\right\}$ is a countable local basis of $x$. Hence $\mathcal{T}$ is first countable. 
\\
Let $x,y \in X$ be such that $x \neq y$. Then $d_\om(x,y)>a$ hence there is $r>0$ such that $d_\om(x,y)>a+r>a$. From $(C_1)$, $\gamma(d_\om(x,y),a+r)>a$. 
Let $r_1=\gamma(d_\om(x,y),a+r)-a$. \\
Claim: $B(x,r_{1}) \cap B(y,r)=\emptyset$.\\
Indeed, let $z \in B(x,r_{1}) \cap B(y,r)$. Then $d_\om(x,z)<a+ r_1=\gamma(d_\om(x,y),a+r)$ and $d_\om(y,z)<a+r$ so $d_\om(y,x) \leq \om(d_\om(y,z),d_\om(z,x)) < \om(a+r,\gamma(d_\om(x,y),a+r)) \leq d_\om(x,y)$ from $(C_1)$, a contradiction, whence the claim.
\\
Given Proposition \ref{topologiez}, to prove the equality $\mathcal{T}=\mathcal{T}_i$ for all $i \in \{1,2,3\}$, it suffices to show that $\mathcal{T} \subset \mathcal{T}_1$. 
Let $A \in \mathcal{T}$ and $x \in A$ There exists $r>0$ such that $B(x,r) \subset A$. If $r'=\om(a,\gamma(a+r,a))$, $r' \in (\im \om) \setminus\{a\}$ since $\gamma(a+r,a)>a$ from $(C_2)$ implies that $r'>\om(a,a) \geq a$ as $\om$ is increasing in both variables. From $(C_1)$, $r'=\om(a,\gamma(a+r,a)) \leq a+r$, hence $|r'-a|=r'-a \leq r$. Thus $B(x,|r'-a|) \subset B(x,r) \subset A$, hence $A \in \mathcal{T}_1$.
\end{proof}

\subsubsection{Generating the O-metric topology and relating properties $(C_i)$, $(E_i)$, $(U_i)$, $i=1,2$}
\begin{theorem}\label{surprise}
The topology of an O-metric space $X$ can be generated by an upward O-metric on $X$. 
\end{theorem} 
\begin{proof}
Let $(X,d_\om,a)$ be an O-metric space, with $d_\om:X \times X \to I_a$ and $\om:I_a \times I_a \to \mathbb{R}_+$, where $I_a$ is an interval in $\mathbb{R}_+$ containing $a \geq 0$. Define the map $\xi:I_a \times I_a \to \mathbb{R}_+$ by:
\begin{equation*}
\xi(u,v)=\max\left\{\om(u,v),\om(u,2a-v),\om(2a-u,v),\om(2a-u,2a-v),2a\right\} ~~ \forall u,v \in [a,\infty).
\end{equation*}
One can check that $(X,d_\xi,a)$ where $d_\xi$ defined by $d_\xi(x,y)=a+|d_\om(x,y)-a|$ for all $x,y \in X$, is an upward metric space. 
Let $\{x_n\}_{n \in \mathbb{N}}$ be a sequence in $X$. We have that:
\begin{equation*}
\begin{array}{lcl}
d_\om(x_n,x) \to a & \Longleftrightarrow & |d_\om(x_n,x)-a| \to 0 \\ & \Longleftrightarrow& d_\xi(x_n,x) =a+|d_\om(x_n,x)-a| \to a.
\end{array}
\end{equation*}
Thus the topology generated by $d_\om$ coincides with that generated by the upward O-metric $d_\xi$.
\end{proof}
\noindent
This result allows us to focus in the remaining part of the paper, on upward O-metric spaces. 
\begin{proposition}
Given an O-metric space $(X,d_\om,a)$ where $I_a = [a,\infty)$,
\begin{enumerate}
\item If $\om$ satisfies $(E_1)$ and $(E_2)$, then it satisfies $(C_1)$  and $(C_2)$;
\item If $\om$ satisfies $(C_1)$ and $(C_2)$, then it satisfies $(U_1)$, $(U_2)$ (and consequently, $(U_2'))$.
\end{enumerate}
\end{proposition}
\begin{proof}
\item
1) Suppose there is a strictly increasing function $\lambda:[a,\infty) \to [0,\infty)$ with $\lambda(a)=0$, such that $\lambda(\om(u,v))=\lambda(u)+\lambda(v)$ for all $u,v \in [a,\infty)$. Then $\om(u,v)=\lambda^{-1}(\lambda(u)+\lambda(v))$ for all $u,v \in [a,\infty)$ so $\om$ is strictly increasing on both variables as composition of two strictly increasing functions. Thus condition $(C_2)$ is satisfied. Define $\gamma:[a,\infty) \times [a,\infty) \to \mathbb{R}$ as $\gamma(u,v)=\lambda^{-1}(\lambda(u)-\lambda(v))$. Suppose that $a \leq u <r$. Then $\gamma(r,u)=\lambda^{-1}(\lambda(r)-\lambda(u)) >\lambda^{-1}(0)=a$ and $\om(u,\gamma(r,u))=\lambda^{-1}(\lambda(u)+\lambda(\gamma(r,u)))=\lambda^{-1}(\lambda(u)+\lambda(r)-\lambda(u))=r$. Thus $(C_1)$ is also satisfied. 
\item
2) Now suppose $\om$ is strictly increasing in both variables and that $\gamma:[a,\infty) \times [a,\infty) \to \mathbb{R}$ is such that $\gamma(r,u)>a$ and $\om(u,\gamma(r,u)) \leq r$ whenever $a \leq u <r$. Given $\epsilon>0$, for all $u \in [a,a+\epsilon)$, from $(C_1)$, $\gamma(a+\epsilon,u) >a$ and $\om(u,\gamma(a+\epsilon,u)) \leq a +\epsilon$, and since $\om$ is strictly increasing in both variables, $\om(a,a)< \om(u,\gamma(a+\epsilon,u)) \leq a+\epsilon$. Thus $a \leq \om(a,a) < a+\epsilon$ for all $\epsilon>0$, hence $\om(a,a)=a$. Hence, $(U_2)$ is satisfied. In fact, $(U_2')$ is also satisfied since $\om(u,a)=a=\om(a,a)$ implies $u=a$, and $\om(a,v)=a=\om(a,a)$ implies $v=a$. For $\epsilon>0$, if $\delta=\min\left\{\frac{\epsilon}{2}, \gamma(a+\epsilon,a+\frac{\epsilon}{2})-a\right\}$, $\delta>0$ since from $(C_1)$, $\gamma(a+\epsilon,a+\frac{\epsilon}{2})>a$ and $\om\left(a+\frac{\epsilon}{2},\gamma(a+\epsilon,a+\frac{\epsilon}{2})\right) \leq a+\epsilon$; for all $u,v \in [a,\infty)$, $|u-a|<\delta$ and $|v-a|<\delta$ imply that $\om(u,v) <\om(a+\delta,a+\delta) \leq \om\left(a+\frac{\epsilon}{2},\gamma(a+\epsilon,a+\frac{\epsilon}{2})\right) \leq a+\epsilon$ hence $|\om(u,v)-a|<\epsilon$. Thus $\om$ is continuous at $(a,a)$ and $(U_1)$ is satisfied.
\end{proof}
\noindent
Multiplicative metric spaces satisfy $(E_1)$ and $(E_2)$ (see proof of Corollary \ref{multiplicativem}), while b-metric spaces satisfy $(U_1)$, $(U_2)$ and $(U_2')$ but not $(C_1)$. In the case where a function $\om$ associated to an O-metric space $(X,d_\om,a)$, where $I_a = [a,\infty)$,  satisfy $(E_1)$ and $(E_2)$, completeness is directly related to the completeness of the associated metric $\lambda \circ d_\om$:
\begin{proposition}
Let $(X,d_\om,a)$ be an O-metric space where $I_a = [a,\infty)$, with $\om:[a,\infty) \times [a,\infty) \to [a,\infty)$, where $a \in \mathbb{R}_+$. Suppose that there is $\lambda:[a,\infty) \to [0,\infty)$ satisfying conditions $(E_1)$ and $(E_2)$ of Theorem \ref{metricequiv}. Then $(X,d_\om,a)$ is complete if and only if the metric space $(X,\lambda \circ d_\om)$ is complete.
\end{proposition}

\section{A fixed point theorem on O-metric spaces}

In this section, we provide a sneak peek at the usefulness of the setting of O-metric spaces by providing a general framework for fixed point theorems of contractive-like mappings in O-metric spaces. The following lemma affirms the equivalence between triangle $\om$-inequalities in O-metrics and polygon inequalities only distorted by a function $\Delta_n$ of the number of the sides of the polygon.\footnote{See next footnote.}

\begin{lem}\label{state}
Let $X$ be a non-empty set $X$, and $d: X\times X\longrightarrow I_a$ is a function, where $I_a = [a, \infty),$ with  $a \in \mathbb{R}_+$.  Then $(X,d,a)$ is an $\om$-metric space,  for some $\om: I_a \times I_a \to \mathbb{R}_+$, if and only if $d$ satisfies the following:
\begin{itemize}
\item[$(i)$]  $\forall x_1,x_2 \in X$,  $d(x_1,x_2)=a \Longleftrightarrow x_1=x_2;$
\item[$(ii)$] $\forall x_1,x_2 \in X$, $d(x_1,x_2)= d(x_2,x_1);$
\item[$(iii)$] For each $n\in \mathbb{N}$, there exists a function $\Delta_n : I_a^n \longrightarrow \mathbb{R}_+$, non-decreasing in all variables, such that for all $x_0,x_1, \ldots, x_n \in X$,
\begin{equation}\label{olal}
d(x_0, x_n) \leq \Delta_n(d(x_0, x_1), d(x_1, x_2), \dots, d(x_{n-1}, x_n)). \footnote{(\ref{olal}) can be referred to as a polygon inequality only distorted by a function $\Delta_n$ of the number of the sides of the polygon.}
\end{equation}
\end{itemize}
\end{lem}

\begin{proof}
If $X$ is an O-metric space, $(i)-(iii)$ hold, with $\Delta_n$ obtained by successive applications of the triangle $\om$-inequality on the terms $d(x_{i},x_{j})$, $0 \leq i<j \leq n$, until each of $d(x_0,x_1)$, $d(x_1,x_2)$, $\ldots$, $d(x_{n-1},x_n)$ appear exactly once in the right hand side of the last inequality.  Conversely, if $(i)-(iii)$ hold, then  $X$ is an $\om$-metric space, with $\om=\Delta_2$.
\end{proof}
\noindent 
Note that if $a=0$ and $\Delta_n(t_1,t_2,\ldots,t_n)=\sum_{i=1}^nt_i$, then (\ref{olal}) reduces to the triangle inequality in metric spaces.
We then have the following theorem which extends the Banach contraction principle in O-metric spaces:

\begin{theorem} \label{app}
Let $(X,d,a)$ be a complete O-metric space, with  $d:X \times X \to I_a,$ where $a \in \mathbb{R}_+$, and $\left\{\Delta_n\right\}_{n \in \mathbb{N}}$ satisfying  conditions $(i)-(iii)$ of Lemma \ref{state}. 
Suppose $T: X \to X$ is a map such that $\forall x, y \in X$, $d(Tx, Ty) \leq \psi (d(x, y))$, where $\psi: I_a \rightarrow I_a$ is a non-decreasing function, continuous at $a$, with $\psi(a) = a$ and:
\begin{itemize}
\item[$(iv)$] $\forall t \in I_a$, $\lim_{n \to \infty} \psi^n(t) =a$;
\item[$(v)$] $\forall i\in \mathbb{N}$, $\lim_{t \to a} \Delta_i \left(\psi(t), \psi^2(t), \ldots, \psi^i(t)\right) = a$.
\end{itemize}
If either $T$ is continuous or $\Delta_2$ is continuous at (a, a) with $\Delta_2(a, a) = a$, then, $T$ has a unique fixed point. 
\end{theorem}

\begin{proof}
Let $x_{0} \in X$ and $\{x_n\}_{n\in\mathbb{N}}$  a sequence such that $x_{n+1} = T{x_n}$ $\forall n\geq 0$. Then, 
\begin{equation*}
\begin{array}{lcl}
d(x_n,x_{n+1}) = d(Tx_{n-1},Tx_n) & \leq & \psi(d(x_{n-1},x_n))
\\
&\leq& \psi^2(d(x_{n-2},x_{n-1}))
\\
&\vdots&
\\
&\leq& \psi^n(d(x_0,x_1)), ~~~~\mbox{ $\forall n\geq 0.$ }
\end{array}
\end{equation*}

\noindent
For all $n, i \in \mathbb{N}$,
\begin{equation*}
\begin{array}{lcl}
d(x_n,x_{n+i})& \leq &\Delta_i (d(x_n, x_{n+1}), d(x_{n+1}, x_{n+2}), \ldots, d(x_{n+i-1}, x_{n+i}))
\\
&\leq & \Delta_i (\psi(t_n), \psi^2(t_n), \ldots, \psi^i(t_n))
\\

& where & t_n =\psi^{n-1} (d(x_{0}, x_{1})).
\end{array}
\end{equation*}
From $(iv)$, $\lim_{n \to \infty}\psi^{n-1}(d(x_0, x_1)) = a,$ hence, from $(v)$, as $t_n \to a,$ 

\begin{equation*}
\lim_{n \to \infty} \Delta_i (\psi(t_n), \psi^2(t_n), \ldots, \psi^i(t_n)) =a. 
\end{equation*}
Thus, $\lim_{n \to \infty} d(x_n, x_{n+i}) = a $ $\forall i \geq 1.$
Therefore, $\{x_n\}$ is a Cauchy sequence and thus converges to some $x^{*} \in X$ (since X is complete).\\
\noindent
\textit{Case 1:} 
If $T$ is continuous, then $x^* = \displaystyle \lim_{n \to \infty}x_{n+1} = \displaystyle \lim_{n \to \infty} Tx_n = Tx^*$ is a fixed point of $T$. 
\\
\textit{Case 2:} Suppose $\Delta_2$ is continuous at $(a, a)$ and that $\Delta_2(a, a) = a$.
\begin{equation*}
\begin{array}{lcl}
d(Tx^*, x^*)& \leq &\Delta_2 (d(Tx^*, Tx_n), d(Tx_n, Tx^*))
\\
& \leq & \Delta_2 (\psi(d(x^*, x_n)), d(x_{n+1},x^*))
\end{array}
\end{equation*}
As $n \to \infty$, $d(Tx^*, x^*) \leq \Delta_2(a,a)=a$, hence $x^* = Tx^*$ is a fixed point of  $T$.
\\
From (v) and the monotonicity of $\psi$, we have that $\psi(t)<t$ for $t>a$. If $z$ is another fixed point of $T$, $d(x^*,z)=d(Tx^*,Tu) \leq \psi(d(x^*,z))$ hence $d(x^*,z)=a$ and $x^*=z$. Therefore, the fixed point of $T$ is unique.
\end{proof}
\noindent
The Banach contraction principle in the setting of metric spaces and b-metric spaces are obtained as corollaries as follows:

\begin{corollary}\label{BCP}\textup{\cite{Banach1922}}
Let (X, d) be a complete metric space, and $T: X \to X$ a map such that $d(Tx, Ty) \leq kd(x, y), \forall x, y \in X,$ for some $k \in [0, 1)$. Then, T has a unique fixed point.
\end{corollary}

\begin{proof}
$(i)-(iii)$ of Lemma \ref{state} hold with $a=0$ and 
$\Delta_n : I_0 \longrightarrow \mathbb{R}_+$ defined by $\Delta_n(\alpha_1, \alpha_2, \ldots, \alpha_n) = \displaystyle\sum^n_{i=1}\alpha_i$, $\forall (\alpha_1, \alpha_2, \ldots, \alpha_n) \in I_0^n.$  \\
If $\psi: T_0 \longrightarrow I_0$ is defined by $\psi(t) = kt$, $\forall t \geq 0$, then $ (iv)$ and $(v)$ of Theorem \ref{app} hold. Indeed,
 $\forall t \geq 0,\displaystyle \lim_{n \to \infty} \psi^n(t) = \displaystyle \lim_{n \to \infty} k^nt = 0$ and, for all $i \in \mathbb{N}$,
\begin{equation*}
\begin{array}{lcl}
\displaystyle \lim_{t \to 0} \Delta_i (\psi(t), \psi^2(t), \ldots, \psi^i(t)) & = & \displaystyle \lim_{t \to 0}(kt+k^2t+\ldots +k^it)\\
& =&\displaystyle \lim_{t \to 0} k(1+k^2+\ldots + k^{i-1})t\\
& =&\displaystyle \lim_{t \to 0} {\frac{k-k^{i+1}}{1-k}} t\\
& =& 0.
\end{array}
\end{equation*}
$d(Tx, Ty) \leq \psi(d(x, y)), \forall x, y \in X,$ and $T$ is continuous (in fact, $\Delta_2$ being the addition operation is continuous at $(0, 0)$, with $\Delta_2(0, 0) =0$). Thus, from Theorem \ref {app}, $T$ has a fixed point.
\end{proof}

\begin{corollary}\textup{\cite{suzuki}}
Let (X, d) be a complete b-metric space, with $s\geq 1,$ and $T: X \longrightarrow X$ a mapping such that for some $k \in (0, 1)$ and for all $x, y \in X,$ $d(Tx, Ty) \leq kd(x, y)$. Then $T$ has a unique fixed point.
\end{corollary}

\begin{proof}
From Lemma 5 in \cite{suzuki}, $d(x_0, x_n) \leq s^{f(n)} \displaystyle \sum_{j=0}^{n-1} d(x_j, x_{j+1})$, $\forall n \in \mathbb{N}$, \\$\forall (x_0, x_1,\ldots, x_n) \in X^{n+1}$, where $f(n) = -[-\log_2n]$ (here, [t] denotes the maximum integer not exceeding a real number $t$). Thus, $(i)-(iii)$ in Lemma \ref{state} hold with $a=0$ and $\Delta_n: I_0^n \longrightarrow \mathbb{R}_+ $ such that $\Delta_n(\alpha_1, \alpha_2, \ldots, \alpha_n) = s^{f(n)} \displaystyle \sum_{j=1}^n \alpha_j.$ \\
It is easy to see that $(iv)$ of Theorem \ref{app} holds, taking $\psi: I_0 \longrightarrow I_0$ such that $\psi(t) = kt$, $\forall t \geq 0$.\\
In fact, if we define $\Delta_n^{\prime}: I_0^n \longrightarrow \mathbb{R}_+$ by 
$$\Delta_n^{\prime} (\alpha_1, \alpha_2, \ldots, \alpha_n) = 
\left\{
\begin{array}{lcl}
s^l \displaystyle \sum_{i=1}^n\alpha_i & \mbox{ if } & n \leq 2^l \\
\displaystyle \sum_{i=0}^{\mu-1}\left(s^{i+l+1} \displaystyle \sum_{j=1+i.2^{l}}^{(i+1)2^l}\alpha_j\right) +s^{\mu} \displaystyle \sum_{i=\mu.2^l+1}^{n}\alpha^i &  \mbox{ if } & n > 2^l,
\end{array}
\right.$$
where $\mu = \left[\frac{n}{2^l}\right],$ with $l \in \mathbb{N}$ chosen such that $ sk^{2^{l}} < 1,$
$(i)-(iii)$ in Lemma \ref{state} and $(iv)$ and $(v)$ in Theorem \ref{app} still hold, with $a = 0$, and $\Delta_n^{\prime}$ substituting $\Delta_n$ (see proof of Lemma 6 in \cite {suzuki}).\\
$\forall x, y \in X$, $d(Tx, Ty) \leq \psi(d(x, y))$,  and $T$ is continuous, $\Delta_2^{\prime}(\alpha_1, \alpha_2) =s(\alpha_1 + \alpha_2)$ is continuous at $(0, 0)$, and $\Delta_2^{\prime}(0, 0) = 0.$ Thus, from Theorem \ref{app}, $T$ has a fixed point.  
\end{proof}


\section{Polygon inequalities and Applications}

\subsection{A more precise triangle inequality} 

The concept of triangle $\om$-inequality in O-metric spaces can be used to provide sharp inequalities for distances in metric spaces. We begin by recalling the following definition:

\begin{definition}\textup{\cite{lines}}
Let $(X,d)$ be a metric space, with three distinct points $a,b,c$. The point $b$ is said to lie between $a$ and $c$ if $d(a,c)=d(a,b)+d(b,c)$, and one writes $[abc]$ when this holds.
\\
The line generated by two points $a,b \in X$ in a metric space, denoted $\overline{ab}$, is defined to be the set consisting  of $a,b$ and all other points $c$ satisfying $[acb]$, $[abc]$ or $[bac]$.
\end{definition}
\noindent
We can now state the following definitions:
\begin{definition}
A metric space $(X,d)$  is said to admit a metric betweenness (or a non-trivial line) if there exists a distinct triple $x,y,z \in X$ such that $d(x,z)=d(x,y)+d(y,z)$. In other words, $X$ is said to admit a non-trivial line if there is a line in $X$ with at least three points. 
\end{definition}

\begin{definition}
A sequence $\{x_n\}$ of distinct points of a metric space $(X,d)$ such that $d(x_1,x_{n+1})=\sum_{i=1}^n d(x_i,x_{i+1})$ for any $n \in \mathbb{N}$, is said to be a metric betweenness sequence. In such case, the points $x_n$ all belong to a line, hence $X$ has a line that is at least countably infinite. 
\end{definition} 

\begin{example}
The Euclidean metric $d(x,y)=\|x-y\|$ in $\mathbb{R}^\textup{N}$ admits a metric betweenness sequence (hence, a metric betweenness). 
\end{example}

\begin{example}
Given a set $X$ with at least three distinct elements, the discrete metric does not admit any metric betweenness.
\end{example}
\noindent
On the other hand, a set $X$ with at least three distinct elements can also be endowed with a metric for which there is a metric betweenness sequence. This is shown in the following lemma:

\begin{lem}
Given any set $X$ with at least three distinct elements, there is a metric $d$ on $X$ that admits a metric betweenness. In fact, if $X$ is infinite, there is some metric $d$ that admits a metric betweenness sequence. 
\end{lem}

\begin{proof}
Suppose $X=\{x_1,x_2,x_3,\ldots\}$ is countable (possibly finite), with $x_1$, $x_2$, $x_3$ all distinct.  The function $d(x_i,x_j)=|i-j|$ is a metric on $X$ and admits a metric betweenness $[x_1 x_2 x_3]$. In the case where $X$ is infinite, the sequence $\{x_n\}$ is a metric betweenness sequence. Suppose now that $X$ is uncountable, with a bijective map $\varphi:X \to \mathbb{R}$. The function $d(x,y)=|\varphi(x)-\varphi(y)|$ is a metric on $X$ and the sequence $\{\varphi^{-1}(n)\}$ is a metric betweenness sequence.
\end{proof}
\noindent
Now, let $(X, \tilde{d})$ be a metric space and $\{x_n\}$ a sequence with no metric betweenness. For each $i \in \mathbb{N}$, the following strict inequality holds 
\begin{equation}
 \tilde{d}(x_1,x_{n+1}) <\sum_{i=1}^n  \tilde{d}(x_i,x_{i+1}).%
\end{equation}
\noindent
In this case, it is interesting to find a more precise triangle inequality. \\
Observe that the metric space $(X, \tilde{d})$ is also an $\tilde{\om}$-metric space, with 
\begin{equation}
\tilde{\om}(u,v)=\sup\{ \tilde{d}(a,c): ~  \tilde{d}(a,b)=u, ~ \tilde{d}(b,c)=v\},
\end{equation} 
since the triangle $\tilde{\om}$-inequality $ \tilde{d}(x,z) \leq \tilde{\om}( \tilde{d}(x,y), \tilde{d}(y,z)) ~ \forall x,y,z \in X$, holds. In fact, $\bar{\om}(u,v) \leq u+v$ for all $u,v \geq 0$. 
\\
\\
It may not be convenient to work with $\tilde{\om}$. We consider the case where $\tilde{d}$ is a composition of some function and a metric.
 Suppose that there is another metric $d$ on $X$ such that $ \tilde{d}=\theta \circ d$, where $\theta:[0,\infty) \to [0,\infty)$ is an increasing  sub-additive map. By simple verification, $(X, \tilde{d})$ is an $\om$-metric space with 
\begin{equation}\label{casecomp}
\om(u,v)= \theta(\theta^{-1}(u)+\theta^{-1}(v)) ~~\forall u,v \geq 0.
\end{equation}
One can check that $\tilde{\om}(u,v) \leq \om(u,v) \leq u+v$ for all $u,v \in [0,\infty)$, with equalities if there is a metric betweenness $[xyz]$ (either for $d$ or $ \tilde{d}$) such that $u$ is the distance between $x$ and $y$, and $v$ the distance between $y$ and $z$. 
\\
A straightforward computation of the triangle $\om$-inequality using (\ref{casecomp}) gives that:
\begin{equation}\label{bettereq}
 \tilde{d}(x_1,x_{n+1}) \leq \theta\left(\sum_{i=1}^n \theta^{-1}(\tilde{d}(x_i,x_{i+1}))\right).%
\end{equation}
With the aid of (\ref{bettereq}), the maximum distance between the end points of a sequence of consecutively equidistant points can be better approximated. Indeed, suppose that for each $i \in \mathbb{N}$, $ \tilde{d}(x_i,x_{i+1})=\alpha$ for some $\alpha >0$. By triangle inequality,
\begin{equation}\label{least}
 \tilde{d}(x_1,x_{n+1}) \leq n\alpha ~~\forall n \in \mathbb{N}.
\end{equation}
However, from (\ref{bettereq}),  
\begin{equation}\label{better}
 \tilde{d}(x_1,x_{n+1}) \leq \theta(n\theta^{-1}(\alpha)) ~~ \forall n \in \mathbb{N}.
\end{equation}
From the properties of $\theta$, (\ref{better}) is a better estimation of $ \tilde{d}(x_1,x_{n+1})$ than (\ref{least}), since, $\theta(n\theta^{-1}(\alpha)) \leq n\alpha$ $~\forall n \in \mathbb{N}$.

\begin{example}
In $\mathbb{R}^\textup{N}$, consider the metric $ \tilde{d}(x,y)=\ln(1+\|x-y\|)$ for all $x,y \in \mathbb{R}^N$, where $\| \cdot\|$ is the Euclidean norm. One can easily check that $ \tilde{d}$ does not admit any metric betweenness sequence, although $\|\cdot\|$ does. Suppose that $\{x_n\}$ is a sequence of points in $\mathbb{R}^N$ at a unit distance $ \tilde{d}$ from each other (i.e. $ \tilde{d}(x_i,x_{i+1})=1$ for each $i \in \mathbb{N}$). By the triangle inequality, we have the estimation $ \tilde{d}(x_1,x_{n+1}) \leq n$ for each $n \in \mathbb{N}$. However, one can use (\ref{better}) with $\theta(t)=\ln(1+t)$ for all $t \geq 0$, to obtain the better estimate: $ \tilde{d}(x_1,x_{n+1}) \leq \ln(1+n(e-1))$ for $n \in \mathbb{N}$. 
\end{example}

\subsection{$s$-relaxed triangle inequalities and triangle $\om$-inequalities in some b-metrics}

\noindent
In a similar way, the triangle $\om$-inequality can be used to provide sharp inequalities for b-metrics on a set $X$ that are compositions of some bijective function $\varphi:\mathbb{R}_+ \to \mathbb{R}_+$ with $\varphi(0)=0$, and a metric on $X$. Such b-metrics are also $\om$-metrics, with $\om(u,v)=\varphi(\varphi^{-1}(u)+\varphi^{-1}(v))$ as seen in  Example \ref{refev}. 
The triangle inequality in b-metric spaces $d(x_1,x_3) \leq s[d(x_1,x_2)+d(x_2,x_3)]$ and the resulting polygon inequality $d(x_1,x_{n+1}) \leq s^{f(n)}\sum_{i=1}^n d(x_i,x_{i+1})$ for points $x_0,x_1,\ldots,x_{n+1} \in X$ are not as sharp as the triangle $\om$-inequality
$d(x_1,x_3) \leq \varphi(\varphi^{-1}(d(x_1,x_2))+\varphi^{-1}(d(x_2,x_3)))$ and
$d(x_1,x_{n+1}) \leq \varphi\left(\sum_{i=1}^n \varphi^{-1}(d(x_i,x_{i+1}))\right)$. 
\\
\\
Consider for example, on $X=\mathbb{R}^2$, the area $d(x,y)=\frac{\pi}{4}\|x-y\|^2$ of circles of 
diameter the segment between $x,y \in X$. We have that $d(x,y)=\varphi(\|x-y\|)$, where $\varphi$ defined by $\varphi(t)=\frac{\pi}{4}t^2$ $\forall t \geq 0$, is a bijection on $\mathbb{R}_+$ such that $\varphi(0)=0$. $(X,d,s)$ is a b-metric space with $s=2$. In fact, $2$ is the optimal value for $s$ in the sense that for collinear points $x_1,x_2,x_3$, with $x_3= x_2+\sqrt{2}=x_1+2\sqrt{2}$, $d(x_1,x_3)=2[d(x_1,x_2)+d(x_2,x_3)]$. Furthermore, $d$ is also an $\om$-metric, with $\om(u,v)=\varphi(\varphi^{-1}(u)+\varphi^{-1}(v))=(\sqrt{u}+\sqrt{v})^2$.
\\
The relaxed triangle inequality $d(x_1,x_3) \leq 2[d(x_1,x_2)+d(x_2,x_3)]$ $~\forall x,y,z \in X$, provides useful 
information that given any triangle with vertices $x_1,x_2,x_3$, the area of the circle with diameter 
a side of the triangle is not greater than two times the sum of the areas of the circles with
the 2 other sides as diameter, with equality occuring when $x_1,x_2,x_3$ are on the same segment, with one
point at a distance $\sqrt{2}$ from each endpoint (see Figure \ref{hallowed}).
\\
However, the triangle $\om$-inequality $d(x_1,x_3) \leq \left(\sqrt{d(x_1,x_2)}+\sqrt{d(x_2,x_3)}\right)^2$ is sharper than the $2$-relaxed triangle inequality and the difference $2(u+v)-(\sqrt{u}+\sqrt{v})^2$ is pictured by the curve $z=2(x+y)-(\sqrt{x}+\sqrt{y})^2$ in Figure \ref{TGR} below.

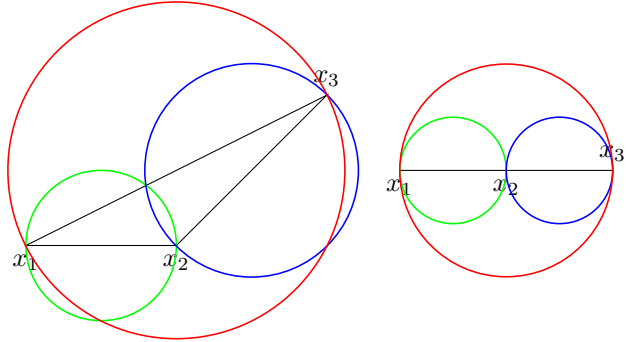
\begin{figure}[h!]\label{hallowed}
\centering
\begin{minipage}{.4\textwidth}
\begin{tikzpicture}
\draw (0,0) node[anchor=north]{$x_1$}
  -- (2,0) node[anchor=north]{$x_2$}
  -- (4,2) node[anchor=south]{$x_3$}
  -- cycle;
\path[draw=green, semithick] (1, 0) circle [radius=1];
\path[draw=blue, semithick] (3, 1) circle [radius=1.42];
\path[draw=red, semithick] (2, 1) circle [radius=2.24];
\end{tikzpicture}
\end{minipage}
\begin{minipage}{.4\textwidth}
\begin{tikzpicture}
\draw (0,0) node[anchor=north]{$x_1$}
  -- (1.415,0) node[anchor=north]{$x_2$}
  -- (2.83,0) node[anchor=south]{$x_3$}
  -- cycle;
\path[draw=green, semithick] (0.708, 0) circle [radius=0.708];
\path[draw=blue, semithick] (2.123, 0) circle [radius=0.708];
\path[draw=red, semithick] (1.415, 0) circle [radius=1.415];
\end{tikzpicture}
\end{minipage}
\caption{Circles of diameter the segment $x_1x_3$ (red), $x_1x_2$ (green) and $x_2x_3$ (blue). The area of the circle in red is not greater than 2 times the sum of the areas of the 2 other circles (left figure), with equality when $x_1,x_2,x_3$ are collinear and $d(x_1,x_2)=d(x_2,x_3)=\sqrt{2}$ (right figure).}
\end{figure}

\begin{figure}[h!]\label{TGR}
\centering
\begin{tikzpicture}
    \begin{axis}
    [
    axis lines=center,
    enlargelimits,
    tick align=inside,
    domain=-5:100,
    samples=200, 
    minor tick num=5,
	colormap/cool,
    ]
    \addplot3 [
	color=blue,
	domain=-5:200,
    samples=50,
	] { 2*(x+y)-(sqrt(x) + sqrt(y))^2 };
    \end{axis}
\end{tikzpicture}
\caption{$z=2(x+y)-(\sqrt{x}+\sqrt{y})^2$.}
\end{figure}
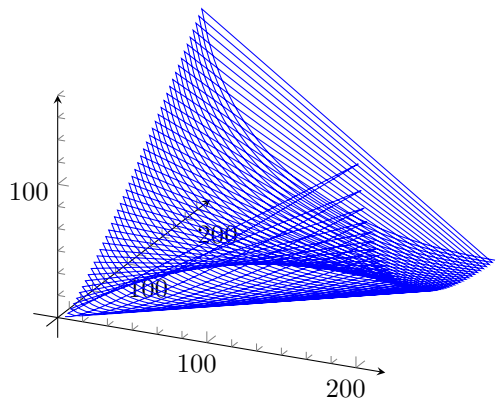

\subsection{$s$-constrained triangle inequality and infinite symmetric matrices}
We note that if the constant $s$ in Definition \ref{(b-Metric)} is taken to be less than $1$ ($s<1$), the space $X$ in the triplet $(X,d,s)$ is a singleton. 
On the other hand, the discrete metric on any non-empty set $X$ satisfies $(d_1)$, $(d_2)$ and the $s$-constrained triangle inequality $d(x,z) \leq s[d(x,y)+d(y,z)]$ for all distinct $x,y,z \in X$, with $s=\frac{1}{2}$, although any infinite sequence $\{x_n\}_{n \in \mathbb{N}}$ of distinct points in the discrete metric is such that $\sum_{i=1}^\infty d(x_i,x_{i+1})=\infty$. We obtain the following proposition which implies that in some sense, metric spaces satisfying the $s$-constrained triangle inequality in every non-degenerate triangle are not ``rich" \footnote{Such metric spaces are not rich in the sense that no infinite sequence of distinct points is a Cauchy sequence as noted in Remark \ref{GodisG}.}
%
%
\begin{proposition}\label{thereisG}
There is no infinite sequence $\{x_n\}$ of distinct points in a metric space $(X,d)$, satisfying for some $s \in [0,1)$ the $s$-constrained triangle inequality, and such that $\sum_{i=1}^\infty d(x_i,x_{i+1})$ converges. 
\end{proposition}
\begin{proof}
Following the proof of Lemma 5 of \cite{suzuki}, we obtain that for each $i \in \mathbb{N}$,
\begin{equation}\label{atlasti}
d(x_i, x_{n}) \leq s^{g(n-i)} \displaystyle \sum_{j=i}^{n-1} d(x_j, x_{j+1}) \leq s^{g(n-i)} \displaystyle \sum_{j=1}^{n-1} d(x_j, x_{j+1}), ~~\forall n \in \mathbb{N}, \end{equation}
where $g(n) = [\log_2(n+1)]+1$ (with [t] denoting the maximum integer not exceeding a real number $t$).  Thus $\lim_{n \to \infty}d(x_i,x_n)=0$ for each $i \in \mathbb{N}$ which means that each term is a limit of $\{x_n\}$, an absurdity.
\end{proof}
\noindent
From Proposition \ref{thereisG}, we note the following:
\begin{remark}[Characteristics of infinite metric spaces satisfying a $s$-constrained triangle inequality in non-degenerate triangles]\label{GodisG}
In an infinite metric space $(X,d)$ (i.e. a metric space with an infinite number of points) satisfying the $s$-constrained triangle inequality for non-degenerate triangles (i.e. $d(x,y) \leq s[d(x,z)+d(z,y)]$ for $x,y,z$ distinct): 
\begin{enumerate}
\item Any infinite sequence $\{x_n\}$ of distinct points must be such that the sequence $\sum_{i=1}^\infty d(x_i,x_{i+1})=\infty$;

\item Any infinite sequence $\{x_n\}$ of distinct points cannot be a Cauchy sequence. Equivalently, any Cauchy sequence must be finite, hence convergent.

\item $X$ is a complete metric space.

\item The longest path between two distinct points is of infinite length, i.e., for any $x,y \in X$ distinct,
$$\sup\left\{\sum_{i=1}^n d(x_{i-1},x_{i}): x_0=x,~ x_n=y,~  x_i \neq x_j \mbox{ for } i \neq j \right\} =\infty.$$
\end{enumerate}
\end{remark}
\noindent
It turns out that the proof of Proposition \ref{thereisG} can be adapted to obtain a parallel result on the maximization of ``triangular quotients" \footnote{Triangular quotients refer to the quotients $\frac{\alpha_{ij}}{\alpha_{ik}+\alpha_{kj}}$, $i\neq k \neq j$, in Proposition \ref{thereisGod}.} in some infinite symmetric matrices:
\begin{proposition}\label{thereisGod}
Let $A=(\alpha_{ij})_{i,j \geq 1}$ be a symmetric infinite matrix with distinct positive real numbers as non-diagonal entries, such that the sum $\sum_{i=1}^\infty \alpha_{i,i+1}$ of the superdiagonal entries is finite, and the triangle inequality $\alpha_{ij} \leq \alpha_{ik}+\alpha_{kj}$ is satisfied for all $i,j,k$. Then
$\displaystyle \sup_{\substack{i,j,k \in \mathbb{N} \\ i \neq k \neq j}} \frac{\alpha_{ij}}{\alpha_{ik}+\alpha_{kj}} =1$.  
\end{proposition}

\begin{example}
Consider the infinite symmetric matrix $A=(\alpha_{ij})_{i,j \geq 1}$ such that for some $r \in (0,1)$,
\begin{equation*}
\alpha_{ij}=\left\{
\begin{array}{lll}
~~1, & \mbox{ if } (i,j)=(1,2)
\\
~~0, & \mbox{ if } i=j
\\
\sqrt{\left(\displaystyle \sum_{q=\floor*{i/2}}^{\floor*{j/2}-1} (-1)^q r^{2q}   \right)^2+\left( \displaystyle \sum_{q=\ceil*{i/2}}^{\ceil*{j/2}-1} (-1)^q r^{2q-1}  \right)^2 }, & \mbox{ if } 2 \leq i <j.
\end{array}
\right.
\end{equation*}
\noindent
Here, $\floor*{\cdot}$ and $\ceil*{\cdot}$ denote the floor function and the ceiling function respectively.\\
We have for all $i,j$ that $\alpha_{ij}=\|x_i-x_j\|$, where $\{x_n\}$ is the sequence in the Euclidean space $\mathbb{R}^2$ such that $x_1=(0,0)$, $x_2=(1,0)$ and for all $i \geq 3$, 
$$x_i=\left( \displaystyle \sum_{q=0}^{\floor*{i/2}-1} (-1)^q r^{2q} , \displaystyle \sum_{q=0}^{\ceil*{i/2}-1} (-1)^q r^{2q-1}\right).$$
From Proposition \ref{thereisG} and \ref{thereisGod}, $\displaystyle \sup_{\substack{i,j,k \in \mathbb{N} \\ i \neq k \neq j}} \frac{\alpha_{ij}}{\alpha_{ik}+\alpha_{kj}} =1$.  

\begin{figure}[h!]
\centering
\begin{tikzpicture}[scale=8]
\draw node[anchor=south] {$x_1$} (0,0) -- (1,0) node[anchor=south] {$x_2$};

\draw (1,0) --  (1,-0.5) node[anchor=west] {$x_3$} ;
\draw (1,-0.5) --  (0.75,-0.5) node[anchor=east] {$x_4$};
\draw (0.75,-0.5) -- (0.75,-0.375) node[anchor=south] {$x_5$};
\draw (0.75,-0.375) -- (0.8125,-0.375) node[anchor=south] {$x_6$};
\draw (0.8125,-0.375) -- (0.8125,-0.40625) node[anchor=west] {$x_7$};
\draw (0.8125,-0.40625) -- (0.796875,-0.40625) node[anchor=north] {$x_8$};
\draw (0.796875,-0.40625) -- (0.796875,-0.3984375);
\draw (0.796875,-0.3984375) -- (0.80078125,-0.3984375);
\draw (0.80078125,-0.3984375) -- (0.80078125,-0.400390625);
\draw[fill] (0,0) circle (.1pt);  
\draw[fill] (1,0) circle (.1pt);  
\draw[fill] (1,-0.5) circle (.1pt);  
\draw[fill] (0.75,-0.5) circle (.1pt);  
\draw[fill] (0.75,-0.375) circle (.1pt); 
\draw[fill] (0.8125,-0.375) circle (.1pt); 
\draw[fill] (0.8125,-0.40625) circle (.1pt); 
\draw[fill] (0.796875,-0.40625) circle (.1pt); 
\end{tikzpicture}
\caption{The terms of the sequence $\{x_n\}$ in the case $r=\frac{1}{2}$, are the endpoints of the segments, starting from the least, clockwise.}
\end{figure}
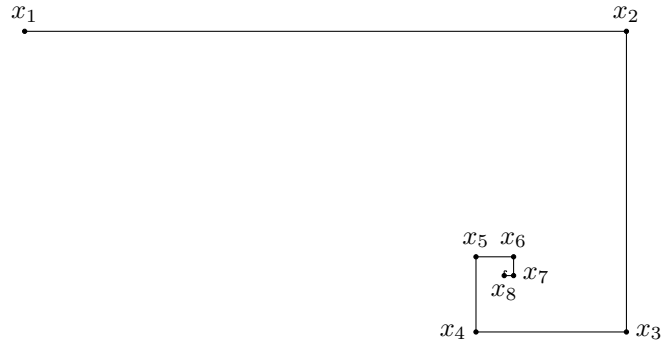
\end{example}

\section*{Declarations}

\section*{Availability of data and materials}
Not Applicable.

\section*{Competing interests}
The authors declare that they have no competing interests.

\section*{Funding}
This research did not receive any specific grant from funding agencies in the public, commercial, or not-for-profit sectors.

\section*{Authors' contributions}
All authors contributed meaningfully to this research work, and read and approved the final manuscript.

\section*{Acknowledgements}
None.






\bibliography{Another2}

\begin{thebibliography}{10}
\expandafter\ifx\csname url\endcsname\relax
  \def\url#1{\texttt{#1}}\fi
\expandafter\ifx\csname urlprefix\endcsname\relax\def\urlprefix{URL }\fi
\expandafter\ifx\csname href\endcsname\relax
  \def\href#1#2{#2} \def\path#1{#1}\fi

\bibitem{Bakhtin1989}
I.~Bakhtin, The contraction mapping principle in almost metric spaces,
  Functional Analysis 30 (1989) 26--30.

\bibitem{Czerwik1993}
S.~Czerwik, Contraction mapping in b-metric spaces, Acta Mathematica et
  Informatica Universitasis Ostraviensis 1 (1993) 5--11.

\bibitem{mult}
A.~Bashirov, E.~Kurpinar, A.Ozyapici, Multiplicative calculus and its
  applications, Journal of Mathematical Analysis and its Applications 337~(1)
  (2008) 36--48.
\newblock \href {http://dx.doi.org/10.1016/j.jmaa.2007.03.081}
  {\path{doi:10.1016/j.jmaa.2007.03.081}}.

\bibitem{mgraph}
M.~Ali, T.~Kamran, A.~Kurdi, Fixed point theorems in b-multiplicative metric
  spaces, U.P.B. Sci. Bull., Series A 79~(3) (2017) 107--116.

\bibitem{survey}
T.~Do{s}enovi\'{c}, M.~Postolache, S.~Radenovi\'{c}, On multiplicative metric
  spaces: survey, Fixed Point Theory and Applications 2016~(92) (2016) 1--17.
\newblock \href {http://dx.doi.org/10.1186/s13663-016-0584-6}
  {\path{doi:10.1186/s13663-016-0584-6}}.

\bibitem{Ige2021}
A.~Ige, H.~Olaoluwa, J.~Olaleru, Some fixed points of multivalued maps in
  multiplicative metric spaces.

\bibitem{thetamet}
F.~Khojasteh, E.~Karapınar, S.~Radenovic, $\theta$-metric spaces: a
  generalization, Mathematical Problems in Engineering 2013~(5) (2013) 1--7.
\newblock \href {http://dx.doi.org/10.1155/2013/504609}
  {\path{doi:10.1155/2013/504609}}.

\bibitem{pbest}
V.~Parvaneh, S.~Ghoncheh, Fixed points of $(\varphi,\phi)$ $\omega$-contractive
  mappings in ordered $p$-metric spaces, Global Analysis and Discrete
  Mathematics 4~(1) (2020) 15--29.
\newblock \href {http://dx.doi.org/10.22128/gadm.2019.290.1019}
  {\path{doi:10.22128/gadm.2019.290.1019}}.

\bibitem{discover}
N.~Mlaiki, H.~Aydi, N.~Souayah, T.~Abdeljawad, Controlled metric type spaces
  and the related contraction principle, Mathematics 6~(10) (2018) 194.
\newblock \href {http://dx.doi.org/10.3390/math6100194}
  {\path{doi:10.3390/math6100194}}.

\bibitem{Olaleru2009}
J.~Olaleru, Some generalizations of fixed point theorems in cone metric spaces,
  Fixed Point Theory and Applications 2009 (2009) 1--10.
\newblock \href {http://dx.doi.org/10.1155/2009/657914}
  {\path{doi:10.1155/2009/657914}}.

\bibitem{Olaoluwa2015}
H.~Olaoluwa, J.~Olaleru, On common fixed points and multipled fixed points of
  contractive mappings in metric-type spaces, Journal of the Nigerian
  Mathematical Society 34~(3) (2015) 249--258.
\newblock \href {http://dx.doi.org/10.1016/j.jnnms.2015.06.001}
  {\path{doi:10.1016/j.jnnms.2015.06.001}}.

\bibitem{Olaoluwa2016}
H.~Olaoluwa, J.~Olaleru, A hybrid class of expansive-contractive mappings in
  cone b-metric spaces, Afrika Matematika 27~(5-6) (2016) 825--840.
\newblock \href {http://dx.doi.org/10.1007/s13370-015-0381-0}
  {\path{doi:10.1007/s13370-015-0381-0}}.

\bibitem{ultra}
A.~V. Rooij, Non-archimedean functional analysis, Marcel Dekker, New York.

\bibitem{openball}
T.~An, L.~Tuyen, N.~Dung, Stone-type theorem on b-metric spaces and
  applications, Topology and its Applications 185-186 (2015) 50--64.
\newblock \href {http://dx.doi.org/10.1016/j.topol.2015.02.005}
  {\path{doi:10.1016/j.topol.2015.02.005}}.

\bibitem{Banach1922}
S.~Banach, Sur les operations dans les ensembles abstraits et leur applications
  aux equations integrales, Fundamenta Mathematicae 3 (1922) 133--181.

\bibitem{suzuki}
T.~Suzuki, Basic inequality on a b-metric space and its applications, Journal
  of Inequalities and Applications 256.

\bibitem{lines}
P.~Aboulker, X.~Chen, G.~Huzhang, R.~Kapadia, C.~Supko, Lines, betweenness and
  metric spaces, Discrete \& Computational Geometry 56 (2016) 427--448.

\end{thebibliography}

\end{document}